\documentclass[a4paper,12pt]{article}
\usepackage[utf8]{inputenc}

\usepackage{geometry}
\usepackage{natbib}
\usepackage{booktabs} 
\usepackage{array} 
\usepackage{paralist} 
\usepackage{verbatim} 
\usepackage{subfig} 
\usepackage{amsfonts,amsmath,amssymb,dsfont}
\usepackage{amsthm}
\usepackage{stmaryrd}
\usepackage{graphicx,hyperref,subfig}
\usepackage{multicol}
\usepackage{color}
\usepackage[toc,page]{appendix}

\makeatletter
\newcommand{\labitem}[2]{%
\def\@itemlabel{#1}
\item
\def\@currentlabel{#1}\label{#2}}
\makeatother

\newtheorem{theorem}{Theorem}[section]
\newtheorem{lemma}[theorem]{Lemma}

\newtheorem{corollary}[theorem]{Corollary}

\title{Non Parametric Hidden Markov Models with Finite State Space: Posterior Concentration Rates}
\author{Elodie Vernet\\Laboratoire de Mathématiques d'Orsay,\\ Univ. Paris-Sud, CNRS, Université
Paris-Saclay,\\ 91405 Orsay, France\\
\texttt{elodie.vernet@math.u-psud.fr}}

\begin{document}

\maketitle

\begin{abstract}
The use of non parametric hidden Markov models with finite state space is flourishing in practice while few theoretical guarantees are known in this framework. Here, we study asymptotic guarantees for these models in the Bayesian framework.  We  obtain posterior concentration rates  with respect to the $L_1$-norm on joint marginal densities of consecutive observations in a general theorem.  We apply this theorem to two cases and obtain minimax concentration rates. We consider discrete observations with emission distributions distributed from a Dirichlet process and continuous observations with emission distributions distributed from Dirichlet process mixtures of Gaussian distributions.
\end{abstract}

\tableofcontents

\section{Introduction}

Hidden Markov models (HMMs) are stochastic processes much used in practice in fields as diverse as genomics, speech recognition, econometrics or climate. A hidden Markov chain is a sequence $(X_t,Y_t)_{t \in \mathbb{N}}$ where the sequence $(X_t)_{t \in \mathbb{N}}$ is a non observed Markov chain  and the sequence of observations $(Y_t)_{t \in \mathbb{N}}$ is a noisy version of the chain $(X_t)_{t \in \mathbb{N}}$.  In this paper we consider the case where the state space of the underlying Markov chain is finite. In this situation, HMMs are often employed to classify dependent data with respect to the hidden states $X_t, ~ t \in \mathbb{N}$. Their popularity is due to their tractability. Since their introduction in \citep{BaPe66}, many algorithms have been developed to infer these models. The books \citep{CaMoRy05}, \citep{macdonald:zucchini:1997} and \citep{macdonald:zucchini:2009} give an overview of this family of models.

Parametric HMMs suffer from a lack of robustness so that non parametric HMMs are used more and more in applications. Indeed two constraints weaken parametric HMMs: the necessary assumption of a bound on the number of states of the Markov chain and the limitations of the parametric modeling of emission distributions (the distributions of an observation $Y_t$ given the hidden states $X_t$).
To deal with these issues,  HMMs with an infinite  countable number of states for the Markov chain are applied in
\citep{BeKr12} to gene expression time course clustering,
\citep{Jo15} to U.S. inflation dynamics and in
\citep{FoJoSuWi09} to segmentation of visual motion capture data.
To handle speaker diarization, \citep{FoSuJoWi11} proposes a model where the number of states of the Markov chain is not bounded and the emission distributions are not restricted to live in a parametric family.
HMMs, where the number of states is known but the emission distributions set is not assumed to be parametric, are used in \citep{LaKnSoDe15} for whales dive modeling, \citep{YaPaRoHo11} for genetic copy number variants, \citep{LaWhMe03} for climate state identification, \citep{Le03} for speech recognition, \citep{GaClRo13} for gene expression identification, see also the references herein. 
This last framework, namely HMMs where the number of states of the Markov chain is known and emission distributions may live in infinite-dimensional sets is the one we consider in this paper. 

 The use of non parametric HMMs is flourishing in practice while few theoretical properties are known. Many theoretical results exist for parametric HMMs particularly for the maximum likelihood estimator, see  \cite{CaMoRy05} and references herein for instance, see also \citep{GuSh08} for a Bernstein von Mises property of the posterior.
In the non parametric framework, there exist few theoretical guarantees of the asymptotic behavior of estimators or posterior since identifiability for general HMMs with finite state space was still an issue until recently.
General identifiability is proved in \citep{GaClRo13} when the number of states of the Markov chain is known and in \citep{GrHa14} when this number is unknown. \citep{GaClRo13} proves that under mild assumptions, the knowledge of the marginal joint density of at least three consecutive observations $(Y_t, Y_{t+1}, Y_{t+2})$ gives the parameters of the HMM up to label switching (i.e. the transition matrix of the Markov chain and the emission distributions). 
Here, we are interested in obtaining asymptotics in the Bayesian framework for the marginal joint density of consecutive observations.

 In the Bayesian non parametric setting, asymptotic analysis typically takes the following two forms: posterior consistency and posterior concentration rates. The posterior is said to be consistent at a parameter $\theta^*$ if it concentrates its mass around $\theta^*$, when the observations come from $\theta^*$ and the number of observations increases. Posterior consistency is related to the merging of posteriors distributions associated to two priors, see \citep{DiFr86}. In a non parametric setup, where it is not feasible to construct a fully subjective prior (on an infinite dimensional space), it  is a minimal requirement, see \citep{GhRa03}. To go further on, one can study the rate at which this concentration occurs. Obtaining a minimax posterior concentration rates  is a criterion of optimality. In particular, minimax concentration rates lead to minimax  Bayesian estimators \citep{GhGhVa00} and to minimax size of credible regions \citep{HoRoSh13}. 
The concentration rate analysis also allows a better understanding of the impact of the prior, see \citep{Rou15} for a discussion.

In Bayesian HMMs where the number of states of the Markov chain is known, \citep{Ve15} provides assumptions leading to posterior consistency for the $L_1$-norm of the marginal density of consecutive observations. Here, we pursue the study of the asymptotic behavior of the posterior distribution in this framework and with the same topology. Namely, we study posterior concentration rates for non parametric HMMs with respect to the $L_1$-norm of the marginal joint density of consecutive observations. 
We first give a general theorem relating the posterior concentration rate to the prior and the true model (Theorem \ref{th1}). Then we apply the theorem to different setups, where we obtain minimax rates (Section \ref{se:applications}). To the best of our knowledge,  these are the first results on posterior concentration rates in non parametric HMMs.

Let us mention the few other asymptotic results we know in the framework of non parametric HMMs. In the non parametric frequentist framework with a finite and known number of states, \citep{DeGaLa15} offers an oracle inequality for a penalized least-squares estimator of the emission distributions.
In the framework of HMMs with an unknown number of states and emission distributions living in a finite-dimensional set,  posterior concentration rates  are studied in \citep{GaRo12}. 
 \citep{GaRo13} proposes asymptotics for the particular case of translated HMMs with finite state space. Finally, convergence  with respect to smoothing distributions is studied in \citep{CaGaCo15}.

The paper is organized as follows. In Section \ref{se:notations}, we precise the studied model and the notations.
In Section \ref{se:genral_theorem}, we state general assumptions under which the posterior concentration rate is derived (Theorem \ref{th1}). 
We have chosen to define a set of assumptions as close as possible to those typically obtained in density estimation for i.i.d. models, see \citep{GhGhVa00}.
 The proof of this theorem is given in Section \ref{se:sketchofproof}. All the other proofs are postponed in the appendices.
Theorem \ref{th1} is applied in Section \ref{se:applications} in two cases. 
In Section \ref{se:discrete}, the observations are assumed to be discrete and the prior on emission distributions is based on a Dirichlet process. We obtain a minimax rate which is $1/\sqrt{n}$ up to a power of $\log n$ in Corollary \ref{co:vit_dis}.
 In Section \ref{se:continuous}, the observations are assumed to be continuous and the emission distributions follow independently Dirichlet process mixtures of Gaussian distribution. 
 Minimax rates of concentration are obtained for Hölder-type functional classes, see Corollary \ref{co:vit_cont}.

\section{Bayesian hidden Markov models and notations}\label{se:notations}

We consider observations coming from homogeneous hidden Markov models with finite state space. Hidden Markov chains are discrete time stochastic processes $(X_t,Y_t)_{t \in \mathbb{N}}$ satisfying the following properties. The sequence $(X_t)_{t \in \mathbb{N}}$ is a Markov chain. Conditionally on the hidden chain $(X_t)_{t \in \mathbb{N}}$, the observations $Y_t$ are independent with $Y_t$ only depending on $X_t$.  The states $(X_t)_{t \in \mathbb{N}}$ are latent, they are called the hidden states. The statistician observes the sequence $(Y_t)_ {t \leq n}$ where $n$ is an integer. Throughout the paper, for any integer $n$, an $n$-uple $(y_{1},\ldots,y_{n})$ is denoted $y_{1:n}$.

We first introduce the notations concerning the Markov chain $(X_t)_{t\in \mathbb{N}}$. 
For all $t\in \mathbb{N}$, $X_t$ belongs to $ \{1, \dots,k\}$, where $k$ is assumed to be known in this paper. A transition matrix $Q$ and an initial probability distribution $\mu$  describe the distribution of the underlying Markov chain
\[
X_1 \sim  \sum_{1\leq i \leq k}\mu_i \delta_i, \quad
X_t | (X_{t-1}=i) \sim  \sum_{1 \leq j \leq k} Q_{i, j} \delta_j
\]
where $\delta_i$ denotes the Dirac measure at $i$.
The set of all initial probability distributions is the $k-1$-simplex $\Delta_k=\{ x \in [0,1]^k : \sum_{1 \leq i \leq k} x_i =1\}$. We denote $\Delta^k_k$ the set of all transition matrices such that each row of the matrix is  an element of $\Delta_k$. 
In the following we need  $\Delta_k(\underline{q})=\{ \mu \in \Delta_k : \mu_i\geq \underline{q} ~ \forall i  \}$
and $\Delta^k_k(\underline{q})=\{ Q \in \Delta_k^k : Q_{i,j}\geq \underline{q} ~ \forall i,j\}$ , for $\underline{q} \in(0,1)$. When $Q$ is in $\Delta^k_k(\underline{q})$, with $\underline{q} \in (0,1)$, then the uniform mixing coefficients, defined in \citep{Ri00}, associated to the corresponding Markov chain are bounded by $\phi(m)\leq (1 - \underline{q} )^m$, moreover the corresponding Markov chain is irreducible and positive recurrent.
 
 The observations $Y_t$ are assumed to live in $\mathbb{R}^d$ which is endowed with its Borel sigma field. The  distribution of $Y_t$ is assumed to be absolutely continuous with respect to some measure $\lambda$ on $\mathbb{R}^d$. Conditionally on $(X_t)_{t \in \mathbb{N}}$, $Y_t$ is distributed from a distribution $f_{X_t} \lambda$ depending on the state $X_t$:
\[
Y_t|(X_s)_{s\in \mathbb{N}} \sim Y_t| X_t \sim f_{X_t} \lambda.
\]
The distributions $f_i \lambda$, $1\leq i \leq k$ are called the emission distributions.
 The set of probability density functions with respect to $\lambda$ is denoted $\mathcal{F}$. The vector $f=(f_1, \dots, f_k)\in \mathcal{F}^k$ is formed with the $k$ emission density functions.

Then the model is completely described by the parameters $\mu$ and $\theta=(Q,f)$ where $\mu \in \Delta^k$ and  $\theta=(Q,f) \in\Delta^k_k \times \mathcal{F}^k=:\Theta$. The model can be visualized in Figure \ref{graph}. 
\begin{figure}[ht]

	\centerline{\includegraphics[width=10cm]{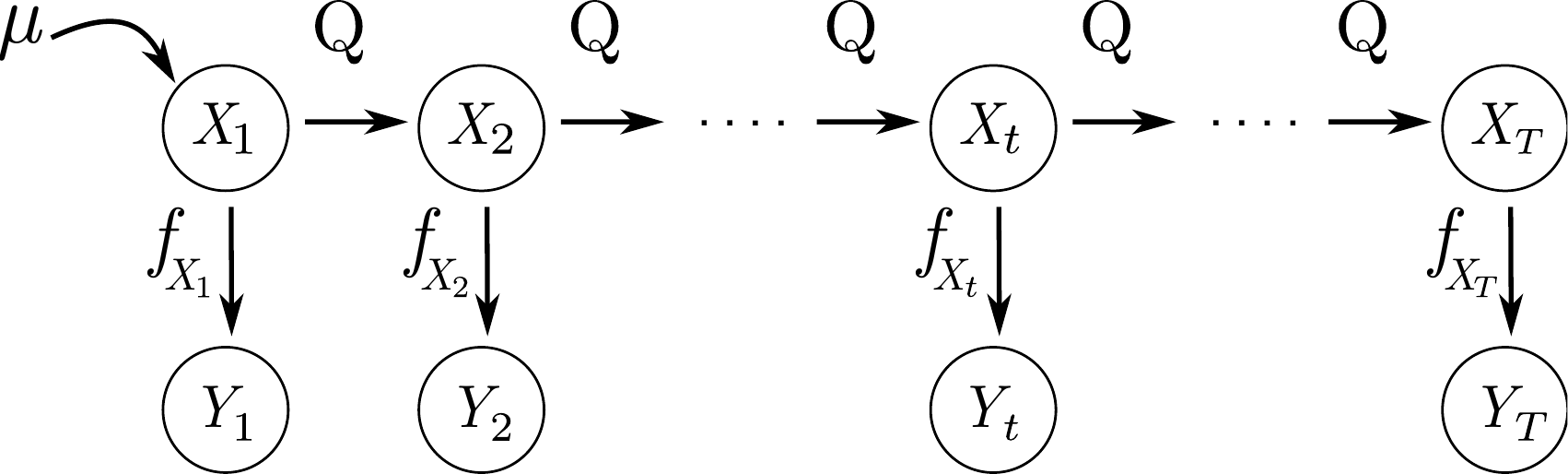}}
	\caption{The model}
	\label{graph}
\end{figure}

Let $P^{\mu,\theta}$ be the probability distribution of the process $(X_t,Y_t)_{t \in \mathbb{N}}$ under $(\mu,\theta)$. 
 Then for any $\theta \in \Theta$, initial probability $\mu$, and 
 measurable set $A$ of $\{1,\dots,k\}^\ell \times (\mathbb{R}^d)^\ell  $, note that:
 \begin{multline*}
 P^{\mu,\theta} ((X_{1:\ell},Y_{1:\ell}) \in A) =\int \sum_{x_1, \dots, x_\ell=1}^k \mathds{1}_{(x_1, \dots,x_\ell,y_1,\dots,y_\ell)\in A} ~ \mu_{x_1} Q_{x_1,x_2} \dots
 Q_{x_{\ell-1},x_\ell} \\
  f_{x_1}(y_1) \dots f_{x_\ell}(y_\ell) \lambda(dy_1)\dots\lambda(dy_\ell) .
 \end{multline*}
 Note that when $Q$ is in $\Delta_k^k(\underline{q})$, with $\underline{q}$ positive, there exists a unique stationary initial distribution $\mu^Q$ associated with $Q$. When $\mu$  is not specified, the stationary distribution associated with the transition matrix $Q$ is considered in the place of $\mu$. In other words, we define $ P^{(Q,f)}:= P^{\mu^Q,(Q,f)}$.
The joint distribution of $\ell$ consecutive observations ($(Y_1,\dots Y_\ell)$ for instance) under the stationary process associated with $\theta$ is denoted $P_\ell^{\theta}$. Let  $p_\ell^{\theta}$ denote the density of $P_\ell^{\theta}$ with respect to $\lambda^{\otimes \ell}$. Then,
\begin{equation*}
  p^{\theta}_\ell (y_1, \dots,y_\ell) 
  = 
 \sum_{x_1, \dots, x_\ell=1}^k 
 \mu_{x_1} Q_{x_1,x_2} \dots Q_{x_{\ell-1},x_\ell} f_{x_1}(y_1) \dots f_{x_\ell}(y_\ell), 
 \quad \lambda^{\otimes \ell} \text{ a.s.} .
 \end{equation*}
The log-likelihood for a sequence of observations $Y_{1:\ell}$ under a parameter $\theta$ is denoted 
\[
L_{\ell}^{\theta}:=\log \big(p_{\ell}^\theta(Y_1, \dots ,Y_{\ell})\big).\] The dependency of $L_{\ell}^{\theta}$ with $Y_{1:\ell}$ is implicit and can be deduced from the context.

Working in the Bayesian framework, we put a prior $\Pi$ on the set of parameters $\Theta$. We choose a product probability measure
$\Pi=\Pi_Q \otimes \Pi_f^{(k)}$ where $\Pi_Q$ is a probability distribution on $\Delta^k_k$ and $\Pi_f^{(k)}$ is a probability distribution on $\mathcal{F}^k$. 
To a realization $\theta$ from $\Pi$, we implicitly associate a stationary initial distribution $\mu^Q$. In other words, we generalize $\Pi$ to a distribution on $\Delta_k \times \Theta$ such that under $\Pi$ and conditionally on $\theta=(Q,f)$, $\mu=\mu^Q$.  Then using the Bayes' theorem, the posterior is expressed by
\[
\Pi( \theta \in A | Y_{1:n}) = \frac{\int_{A} p_n^\theta(Y_{1:n})\Pi(d\theta)}{\int_{\Theta} p_n^\theta(Y_{1:n})\Pi(d\theta)}
.\]

We are interested in the asymptotic behaviour of the posterior that is to say when the number $n$ of observations $Y_{1:n}$ tends to infinity.
For this purpose, we take a frequentist point of view, assuming that the observations come from the true parameters $\mu^*$ and $\theta^*=(Q^*,f^*)$. 
We  suppose that the true initial distribution $\mu^*$ is stationary. 
We also assume that there exists $\underline{q}^*>0$ such that  
\begin{equation}\label{hyp:q*}
Q^* \in \Delta^k_k(\underline{q}^*).
\end{equation}

\citep{Ve15} shows  posterior consistency at $\theta^*$ under general assumptions. In this paper, we consider posterior concentration rates at $\theta^*$.
Recall that the posterior is said to concentrate at rate $\epsilon_n$, a sequence decreasing to $0$, for the loss $D(\cdot,\cdot)$ if there exists a constant $M>0$ such that
\[
\Pi(\theta : D(\theta,\theta^*) \geq  M\epsilon_n | Y_{1:n}) =o_{P^{\theta^*}}(1),
\]
where $Z=o_{P^{\theta^*}}(1)$ means that $Z$ converges in probability to $0$.
 We choose to study the concentration of the posterior from the density estimation point of view. We compare two parameters $\theta$ and $\tilde{\theta}$ by computing the $L_1$-distance between the joint densities $p_\ell^\theta$ and $p_\ell^{\tilde{\theta}}$. For two distributions $P_1$ and $P_2$, let $p_1$ and $p_2$ be their respective densities with respect to a dominated measure $\nu$. The $L_1$-metric is defined by
\[
\lVert p_1 -p_2 \rVert_{L_1(\nu)} = \int \lvert p_1 - p_2 \rvert \nu
\]
and let 
\[
KL(p_1,p_2)=\int p_1 \log\left( \frac{p_1}{p_2} \right) \nu
\]
be the Kullback-Leibler divergence between $p_1$ and $p_2$.
For an integer $\ell \geq 1$, we use the pseudo-distance $D_\ell$ on $\Theta$ defined by 
\[
D_\ell(\theta,\tilde{\theta}) = \lVert p_\ell^\theta - p_\ell^{\tilde{\theta}} \rVert_{L_1(\lambda^{\otimes l})}
.
\]
We study the posterior rate of concentration with respect to this pseudo-distance $D_\ell$.
On $\mathcal{F}^k$, we use the distance $d(\cdot,\cdot)$ such that for all $(f,\tilde{f}) \in (\mathcal{F}^k)^2$
\[
d(f,\tilde{f})=\max_{1 \leq i \leq k} \lVert f_i - \tilde{f}_i \rVert_{L_1(\lambda)}, 
\] 
on $\mathbb{R}^d$, $ d \geq 2$,  we use the supremum norm $\lVert \cdot \rVert$.
For a positive real $\epsilon$, a pseudo distance $D$ defined on a set $A$, let $N(\epsilon,A,D)$ be the covering number that is to say  the minimum number of balls of radius $\epsilon$ (in the pseudo-distance $D$) needed to cover $A$. 
Throughout the paper the notation $\lesssim$ means less or equal up to a multiplicative constant which is not important in the context.

\section{General Theorem}\label{se:genral_theorem}

\subsection{Assumptions and main theorem}\label{se:assumptions_th}

In this section, we state the general Theorem \ref{th1} which gives posterior concentration rates with respect to the $D_\ell$ pseudo-metric.  As in \citep{GhGhVa00} for instance, we propose a set of conditions which relates the rate $\epsilon_n/ \underline{q}_n$ to the prior and the true model. We apply this theorem to the case of discrete observations in Section \ref{se:discrete} and to the case of continuous observations in Section \ref{se:continuous} where minimax rates are achieved.
Now, we enumerate the assumptions of Theorem \ref{th1}.
Assumptions \eqref{hyp:KL} and \eqref{hyp:Fn} concern the prior on the emission distributions $\Pi_f^{(k)}$ and the  vector of the true emission distributions $f^*$. Assumptions \eqref{hyp:voisQ*} and \eqref{hyp:qn} involve the prior on transition matrices $\Pi_Q$ and the true transition matrix $Q^*$.

 For a given sequence $\epsilon_n>0$ 
 tending to $0$, 
 we introduce a sequence $\tilde{\epsilon}_n$ such that $\tilde{\epsilon}_n\leq\epsilon_n$ and the sequence $u_n$ of positive numbers such that
\begin{equation}\label{eq:def_un}
\text{
\begin{minipage}{30em}
\begin{enumerate}[(i)]
\item $u_n=1$, for all $n \in \mathbb{N}$; if $\tilde{\epsilon}_n \gtrsim n^{-s}$, for some $s<1/2$,
\item  $u_n=(\log(n))^{3/2}$, for all $n \in \mathbb{N}$; otherwise.
\end{enumerate}
\end{minipage}
}
\end{equation}
  We consider the following assumptions
\begin{enumerate}[(A)]
\item\label{hyp:KL} there exist a positive constant $C_f$ and a sequence $B_n$ of subsets of $\mathcal{F}^k$ 
such that 
\[
\Pi_f^{(k)} (B_n) \gtrsim \exp(-C_f n\tilde{\epsilon}_n^2)
 \]
 and such that for all $f \in B_n$,

\begin{equation}\label{hypa:log2}
 \int f^*_i(y)  \log^2\left(\frac{f^*_j(y)}{f_j(y)}\right) \lambda(dy) \leq \frac{\tilde{\epsilon}^{2}_n}{u_n} ,
\text{ for all } 1 \leq i,j \leq k,
\tag{A.1}
\end{equation}
there exist  a set $S\subset \mathcal{Y}$ and functions $\tilde{f}_1, \dots \tilde{f}_k$ satisfying
\begin{equation}\label{hypa:qui2}
 \int_S   
  \frac{  \lvert f^*_j(y) - f_j(y)  \rvert^2}{f^*_j(y)}  \lambda(dy) \leq \frac{\tilde{\epsilon}^{2}_n}{u_n}, \text{ for all } 1 \leq j \leq k,
 \tag{A.2}
\end{equation}
\begin{equation}\label{hypa:maxSc}
 \int_{S^c} \tilde{f}_j(y) \lambda(dy) \leq   \tilde{\epsilon}^{2}_n,
\text{ for all } 1 \leq j \leq k,
\tag{A.3}
\end{equation}
\begin{equation}\label{hypa:maxf*Sc}
 \int_{S^c} f^*_j(y) \lambda(dy) \leq \frac{\tilde{\epsilon}^{2}_n}{u_n},
\text{ for all } 1 \leq j \leq k,
\tag{A.4}
\end{equation}
\begin{equation}\label{hypa:maxloSc}
\int_{S} f^*_i(y) \max_{1 \leq j \leq k}
\log \left(\frac{\tilde{f}_j(y)}{f_j(y)}\right)\lambda(dy) \leq   \tilde{\epsilon}^{2}_n,
\text{ for all } 1 \leq i \leq k,
\tag{A.5}
\end{equation} 
\begin{equation}\label{hypa:qui2S}
 \int_S  
  \frac{  \lvert f^*_j(y) - \tilde{f}_j(y)  \rvert^2}{\tilde{f}_j(y)} \lambda(dy) \leq \tilde{\epsilon}^{2}_n.
\text{ for all } 1 \leq j \leq k,
\tag{A.6}
\end{equation}
\item\label{hyp:Fn} there exist positive constants $C$ and $C'$ and a sequence $(\mathcal{F}_n)_{n\geq 1}$ of subsets of $\mathcal{F}^k$ such that 
\[
\Pi_f^{(k)}(\mathcal{F}_n^c) = o(\exp(- C n \tilde{\epsilon}_n^2 )) 
\]
 and 
\[
N\left(\frac{\epsilon_n}{12}, \mathcal{F}_n, d \right) \lesssim \exp\left( C' n \epsilon_n^2
\right),  
\]
\item \label{hyp:voisQ*} there exists a positive constant $C_Q$ such that $C_Q+C_f+C_K < C$ with $C_K= 4 + \log\left(2k/\underline{q}^*\right) + 10^4k^2/{\underline{q}^*}^5 $,
\[
\Pi_Q \left( \left\{ Q : \lVert Q - Q^* \rVert \leq \frac{\tilde{\epsilon}_n}{\sqrt{u_n}}  \right\} \right) \gtrsim \exp( -C_Q n \tilde{\epsilon}_n^2 ),
\]
\item \label{hyp:qn} there exists a sequence $\underline{q}_n$ of $(0,1/k]$ such that
\[
\Pi_Q\Big(\big(\Delta_k^k(\,\underline{q}_n)\big)^c\Big)= o (\exp( -  C n \tilde{\epsilon}_n^2 ))
.\]
\end{enumerate}

Under the above assumptions, we prove that the posterior distribution concentrates at the rate $\epsilon_n/\underline{q}_n$.

\begin{theorem}\label{th1}
Let $\epsilon_n\geq \tilde{\epsilon}_n>0$ be two sequences tending to $0$ such that $n \tilde{\epsilon}_n^2  \to  + \infty$.
Assume \eqref{hyp:KL}, \eqref{hyp:Fn}, \eqref{hyp:voisQ*} and \eqref{hyp:qn}.

Then there exists a positive constant $M$  
such that
\begin{equation}\label{eq:ccl_th}
\Pi \left(\theta : D_\ell(\theta, \theta^*) \geq M \frac{\epsilon_n}{\underline{q}_n}  ~  \bigg| Y_{1:n} \right)= o_{\mathbb{P}^{\theta^*}}(1).
\end{equation}
\end{theorem}

We now discuss Assumptions \eqref{hyp:KL} to \eqref{hyp:qn}.
We have purposely considered assumptions which are as similar as possible to those considered in the set up of density estimation with i.i.d. observations see 
\cite{GhGhVa00}. In particular Assumption \eqref{hyp:Fn} is the typical entropy assumption on a sieve which captures most of the prior mass (see Assumptions (2.2) and (2.3) of \cite{GhGhVa00}). Assumption \eqref{hyp:KL} is slightly more involved than the Kullback-Leibler condition (2.4) of \cite{GhGhVa00} in the case of i.i.d. observations.

This paragraph explains the differences between Assumption \eqref{hyp:KL} and condition (2.4) of \cite{GhGhVa00} and can be omitted at first reading. Following \cite{GhVa07niid}, posterior rates of convergence are obtained by controlling ``Kullback-Leibler" neighborhoods and constructing some tests.
 Under \eqref{hyp:KL} and condition (2.4) of \cite{GhGhVa00}, the Kullback-Leibler neighborhoods are controlled.
 Recall that in \cite{GhGhVa00}, the control of $\mathbb{E}^{\theta^*}(L_n^{\theta^*}-L_n^{\theta})$ and $Var^{\theta^*} (L_n^{\theta^*}-L_n^{\theta})$
is obtained by controlling 
\[ \int f^*_i(y) \log \left(\frac{f^*_i(y)}{f_i(y)}\right) \lambda (dy)
\qquad \mbox{ and } \qquad
\int f^*_i(y) \log^2 \left(\frac{f^*_i(y)}{f_i(y)}\right) \lambda (dy).
\]
Here $\mathbb{E}^{\theta^*}(L_n^{\theta^*}-L_n^{\theta}) \lesssim n \tilde{\epsilon}_n^2$ and $Var^{\theta^*} (L_n^{\theta^*}-L_n^{\theta} )\lesssim n \tilde{\epsilon}_n^2 \log n$  if $f$ and $f^*$ satisfy Assumptions \eqref{hypa:log2}--\eqref{hypa:qui2S} and $\lVert Q - Q^* \rVert \leq \tilde{\epsilon}_n/\sqrt{u_n}$.
 The non intuitive part of \eqref{hyp:KL} comes from the introduction of $\tilde{f}_j$ as an approximation of $f^*_j$, which may be different from $f_j$ in \eqref{hypa:maxSc}, \eqref{hypa:maxloSc} and \eqref{hypa:qui2S}. Indeed, whithout the introduction of $(\tilde{f}_{j, 1\leq j \leq k},S)$, the HMM structure of the likelihood would lead to a crude upper bound of $\mathbb{E}^{\theta^*}(L_n^{\theta^*}-L_n^{\theta})$  of the form 
\[
n\int f^*_i(y) \max_{j} \log \left( \frac{f^*_j(y)}{f_j(y)} \right) \lambda(dy),
\]
see Equations \eqref{eq:contKL} and \eqref{eq:majoration_frac_q_theta} in the proof of Lemma \ref{lemma2v3}.
Since $i$ may be different from $j$ we would loose the local quadratic approximation of the Kullback-Leibler divergence and we would only obtain $n\tilde{\epsilon}_n$ instead of $n\tilde{\epsilon}_n^2$ as an upper bound of $\mathbb{E}^{\theta^*}(L_n^{\theta^*}-L_n^{\theta})$. The details of the control of the Kullback-Leibler  divergence are given in Appendix \ref{se:prKLv3}.

Not withstanding the technical aspects discussed above, Assumptions \eqref{hypa:log2}--\eqref{hypa:qui2S} are verified using  techniques similar to those used in the case of density estimation with i.i.d. observations to control 
\[ \int f^*_i(y) \log \left( \frac{f^*_i(y)}{f_i(y)} \right) \lambda (dy)
\qquad \mbox{ and } \qquad
\int f^*_i(y) \log^2 \left( \frac{f^*_i(y)}{f_i(y)} \right) \lambda (dy).
\]
For $\Pi^{(k)}_f=(\Pi_f)^{\otimes k}$, and many families of individual prior models $\Pi_f$ on the $f_j$'s, the rate obtained by bounding $\max_j KL(f^*_j,f_j)$ in the i.i.d. set up will be the same as in our setup.
For instance, if $\mathcal{Y}=[0,1]$ and $f^*$ is bounded from below and above, a control of $\lVert f_j - f^*_j\rVert_{\infty}^2 \leq \tilde{\epsilon}_n^2$ or $\lVert f_j - f^*_j\rVert_{2}^2 \leq \tilde{\epsilon}_n^2$ and $f_j>c$ imply \eqref{hypa:log2}--\eqref{hypa:qui2S}. This kind of controls have been derived under (hierarchical) Gaussian process priors or log linear priors as in \cite{VaZa09}, \cite{ArGaRo13} and \cite{RiRo12}. Condition \eqref{hyp:KL} becomes more involved when $\mathcal{Y}$ is not compact. This case is treated under non parametric Gaussian mixtures in Section \ref{se:continuous} and in the case where $\mathcal{Y}=\mathbb{N}$ in Section \ref{se:discrete}.

 Assumption \eqref{hyp:voisQ*} is checked as soon as $\Pi_Q$ admits a positive density with respect to the Lebesgue measure which is continuous at $Q^*$ and $\tilde{\epsilon}_n \geq \sqrt{\log(n)/n}$.  The rate $\epsilon_n$ is often equal to $\tilde{\epsilon}_n$ up to $\log n$, the use of these two different rates is usual and allows more flexibility.  Then the rate $\epsilon_n$  is only determined by the non parametric part of the model, i.e. $\Pi_f^{(k)}$ and $f^*$, as described above. 
 
Following the previous explanation, when $\Pi_Q$, $\Pi_f^{(k)}$, $f^*$ and $Q^*$ are fixed, $\epsilon_n$ is specified by Assumption \eqref{hyp:KL} and \eqref{hyp:Fn}. This rate $\epsilon_n$ is deteriorated via $\underline{q}_n$ which is set through Assumption \eqref{hyp:qn}. The larger $\tilde{\epsilon}_n$ is, that is to say the more difficult the estimation of the non parametric part ($f^*$ with $\Pi_f^{(k)}$) is, the more stringent Assumption \eqref{hyp:qn} is. To avoid too small $\underline{q}_n$ which leads to deteriorated posterior convergence rate $\epsilon_n/\underline{q}_n$, one may choose a prior $\Pi_Q$ which is supported on $\Delta_k^k(\underline{q})$ for some $0<\underline{q}\leq \underline{q}^*$. More examples of distribution $\Pi_Q$ are given in Section \ref{se:applications}.

In the following section, we give the proof of Theorem \ref{th1}.


\subsection{Proof of Theorem \ref{th1}}\label{se:sketchofproof}

To obtain posterior concentration rates in the framework of HMMs with finite state space, we  use the technique of proof of  \cite{GhVa07niid}. The key tools of this technique are a control of the prior mass on log-likelihood neighborhoods of $\theta^*$ and the existence of certain tests.
We use the tests built in \cite{GaRo12}.  The main difficulty of the proof arises from the control of log-likelihood neighborhoods.
 These neighbourhoods are controlled  thanks to Lemmas \ref{lemma2v3} and \ref{lemma1}. The proof of these lemmas are based on refinements of  results of \cite{DoMa01} and \cite{Do04}.

In Lemma \ref{lemma2v3}, we control $KL(p_n^{\theta^*},p_n^{\theta})$. Its proof is given in Section \ref{se:prKLv3}.

\begin{lemma}\label{lemma2v3}
 Let $0<\tilde{\epsilon}_n$ be small enough.
Assume that $\theta=(Q,f)\in \Theta$, is such that 
Assumptions  \eqref{hypa:log2}--\eqref{hypa:qui2S} hold with $u_n=1$ for all $ n \in \mathbb{N}$
and 
\begin{equation}\label{hyp:voisQ*2}
\lVert Q - Q^* \rVert \leq \tilde{\epsilon}_n.
\end{equation}

 Then there exists $N>0$ such that for all $n\geq N$, 
 \begin{equation*}
 \begin{split}
 \mathbb{E}^{\theta^*}(L_n^{\theta^*}-L_n^{\theta}) =KL&(p_n^{\theta^*},p_n^\theta) 
\leq C_K n \tilde{\epsilon}_n^2,
\end{split}
\end{equation*}
 where $C_K$ is defined in Assumption \eqref{hyp:voisQ*}.
\end{lemma}
As can be seen in the proof, Assumption \eqref{hypa:log2} can be replaced in Lemma \ref{lemma2v3} by the following weaker assumption:
\begin{equation}\label{hypa:maxloS}
\int_{S^c} f^*_i(y) \max_{1 \leq j \leq k}
\log \left(\frac{f^*_j(y)}{f_j(y)}\right)\lambda(dy) \leq C_K \tilde{\epsilon}^{2}_n,
\text{ for all } 1 \leq i \leq k
\end{equation}
which is implied by Assumptions \eqref{hypa:log2} and \eqref{hypa:maxf*Sc}. Note however that Assumption \eqref{hypa:log2} is used in the control of the variance $Var^{\theta^*} (L_n^{\theta^*}-L_n^{\theta})$ in Lemma \ref{lemma1}.

Lemma \ref{lemma1} gives a control of $Var^{\theta^*} (L_n^{\theta^*}-L_n^{\theta})$.
It is proved in Appendix \ref{se:preuve_KL2}.

\begin{lemma}\label{lemma1} Let $0<\epsilon_n$ be small enough and $u_n$ be defined by Equation \eqref{eq:def_un}
Assume that $\theta=(Q,f)\in \Theta$, is such that  Assumptions 
 \eqref{hypa:log2}, \eqref{hypa:qui2} and \eqref{hypa:maxf*Sc} hold 
 and 
\begin{equation}\label{hyp:voisinage_Q*}
\lVert Q - Q^* \rVert \leq \frac{\tilde{\epsilon}_n}{\sqrt{u_n}} .
\end{equation}
 Then there exists a positive constant $C_{KL^2}$ such that for all $\alpha \in (0,1)$ and $n\in \mathbb{N}$
\begin{equation*}
\begin{split}
Var^{\theta^*}(L_n^{\theta^*}-L_n^{\theta} )&=
\mathbb{E}^{\theta^*}  \left[
\left(\log \frac{p^{\theta^*}(Y_{1:n})}{p^{\theta}(Y_{1:n})}
       -\mathbb{E}^{\theta^*}\left(\log \frac{p^{\theta^*}(Y_{1:n})}{p^{\theta}(Y_{1:n})}\right)\right)^2\right]\\
         & \leq C_{KL^2} \frac{n}{\alpha} \left(\frac{\tilde{\epsilon}_n}{\sqrt{u_n}}\right)^{2-\alpha}.
\end{split}
\end{equation*}
\end{lemma}

We now give the proof of Theorem \ref{th1}.

\begin{proof}[Proof of Theorem \ref{th1}]
This proof follows the lines of the proof of Theorem 1 of \cite{GhVa07niid} with two variants. These differences come from the tests (see Equations \eqref{eq:eq:proofth1_test*} and \eqref{eq:eq:proofth1_test}) and the control of the Kullback-Leibler neighborhoods (Equation \eqref{eq:Bntheta*}).

Using the tests built in the proof of Theorem 4 in \cite{GaRo12},  for all $M>0$, there exists $\psi_n\in [0,1]$ such that
\begin{equation}\label{eq:eq:proofth1_test*}
\mathbb{E}^{\theta^*}(\psi_n) \leq N \left(  \frac{\epsilon_n}{12},    \Delta_k^k( \, \underline{q}_n) \times \mathcal{F}_n  , D_\ell \right) \exp\left( \frac{-n \epsilon_n^2 {\underline{q}^*}^2 k^4 M^2}{128 l} \right)
\end{equation}
and 
\begin{equation}\label{eq:eq:proofth1_test}
\sup_{\substack{\theta \in \Delta_k^k(\,\underline{q}_n) \times \mathcal{F}_n \\D_\ell(\theta,\theta^*)\geq M  \epsilon_n / \underline{q}_n }} 
\mathbb{E}^{\theta}(1-\psi_n) 
\leq \exp\left(  -\frac{n \epsilon_n^2 k^2 M^2 }{128 l}  \right)
.
\end{equation}
Since 
\[
D_\ell(\theta, \tilde{\theta}) \leq \sum_{1 \leq i \leq k} \lvert \mu_i^Q -\mu_i^{\tilde{Q}}\rvert + k(\ell-1) \max_{1 \leq i,j \leq k} \lvert Q_{i,j} - \tilde{Q}_{i,j} \rvert + \ell d(f,\tilde{f}),
\]
we obtain
\begin{equation}\label{eq:eq:proofth1_entropie}
N \left(  \frac{\epsilon_n}{12},    \Delta_k^k( \, \underline{q}_n) \times \mathcal{F}_n  , D_\ell \right)
\leq  \left( \frac{24\ell k(k-1)}{\epsilon_n} \right)^{k(l-1)} 
N \left(  \frac{\epsilon_n}{24 \ell},     \mathcal{F}_n  , d \right)
\end{equation}
which leads to
\begin{equation}\label{E-theta*-phin}
\mathbb{E}^{\theta^*}(\psi_n) \lesssim \exp(-n \epsilon_n^2 (M^2\tilde{C} - C'))
\end{equation}
for some constant $\tilde{C}$, using Assumption \eqref{hyp:Fn}.
We replace Equation (8.4) of \cite{GhVa07niid} by Equation \eqref{E-theta*-phin}. Equation (8.5) of \cite{GhVa07niid} is replaced by Equation \eqref{eq:eq:proofth1_test}.

Let $\alpha_n$ be a sequence tending to $0$, to be specified later. We define
\begin{multline}\label{eq:Bntheta*}
\mathcal{B}_n(\theta^*)\\
:= 
 \left\{ 
\theta :
KL(p_n^{\theta^*},p_n^{\theta}) \leq C_K n \tilde{\epsilon}_n^2 , ~
Var^{\theta^*}(L_n^{\theta^*}-L_n^{\theta})
\leq \frac{C_{KL^2} n }{\alpha_n} \left(\frac{\tilde{\epsilon}_n}{\sqrt{u_n}}\right)^{2-\alpha_n}
 \right\} 
\end{multline}
in the place of  $B_n(\theta^*,\bar{\epsilon}_n,2)$ in the notation of \cite{GhVa07niid},
setting $\bar{\epsilon}_n=\sqrt{C_{K}}\tilde{\epsilon}_n$.
Using Assumptions \eqref{hyp:KL} and \eqref{hyp:voisQ*}, and Lemmas \ref{lemma2v3} and \ref{lemma1}
\begin{equation*}
\Pi\left(\mathcal{B}_n\left(\theta^*\right)\right)  \gtrsim \exp(-(C_Q+C_f) n \tilde{\epsilon}_n^2) \gtrsim \exp(-C n \tilde{\epsilon}_n^2).
\end{equation*}

By choosing
\begin{enumerate}[(i)]
\item $\alpha_n=1/\log(n)$, for all $n \in \mathbb{N}$; if $\tilde{\epsilon}_n\gtrsim n^{-s}$ for some $s<1/2$,
\item $\alpha_n=\log(\log(n))/\log(n)$, for all $n \in \mathbb{N}$; otherwise,
\end{enumerate} 
and following the lines of the proof of Lemma 10 of \cite{GhVa07niid},
\begin{equation}\label{eq:P_theta*_D_n}
 P^{\theta^*} \big(\mathcal{D}_n< \Pi(B_n(\theta^*)) \exp(-2 C_K n \tilde{\epsilon}_n^2)\big) =O\left( \frac{n (\tilde{\epsilon}_n/\sqrt{u_n})^{2-\alpha_n}}{ \alpha_n  (n\tilde{\epsilon}_n^{2})^2} \right)=o(1),
\end{equation}
with $\mathcal{D}_n=\int_{\mathcal{B}_n(\theta^*)} p_n^\theta(Y_{1:n})/p_n^{\theta^*}(Y_{1:n}) \Pi(d\theta)$.
Following the lines of the proof of Theorem 1 of \cite{GhVa07niid}  with the above modifications, we obtain,
\begin{equation}
\begin{split}
&\mathbb{E}^{\theta^*}\left( \Pi\left( \theta \in \Delta_k^k(\,\underline{q}_n) \times \mathcal{F}_n  : D_\ell(\theta, \theta^*) \geq M \frac{\epsilon_n}{\underline{q}_n}  \big| Y_{1:n} \right) \right)
=o(1)
 \end{split}
\end{equation} for $M$ large enough.

This concludes the proof since 
\[\mathbb{E}^{\theta^*}\left(\Pi((  \Delta_k^k(\,\underline{q}_n) \times \mathcal{F}_n )^c | Y_{1:n})\right)=o(1)
\]
using Equation \eqref{eq:P_theta*_D_n} and Lemma 1 of \cite{GhVa07niid}  with 
\[
\Pi((  \Delta_k^k(\,\underline{q}_n) \times \mathcal{F}_n )^c)=o( \exp(-2n \bar{\epsilon}_n^{2})  \Pi\left(  \mathcal{B}_n\left(\theta^*\right)\right)
\]
as soon as $C>2 C_K +C_Q + C_f$ (using Assumptions \eqref{hyp:Fn} and \eqref{hyp:qn}).
\end{proof}

\section{Applications}\label{se:applications}

In this section, we apply Theorem \ref{th1} to different priors and different classes of emission  density functions.
In all examples treated in Section \ref{se:applications}, the prior on emission distributions is chosen to be a product of a distribution $\Pi_f$ on $\mathcal{F}$:
\begin{equation}\label{eq:pif_produit}
\Pi_f^{(k)} = (\Pi_f)^{\otimes k}.
\end{equation}
However, Theorem \ref{th1} can also be applied to other priors such as priors restricted to translated emission density functions, the translation HMM is described in Equation \eqref{eq:HMM_translatee}.

 In Section \ref{se:discrete}, we consider discrete observations, i.e. $\mathcal{Y}=\mathbb{N}$. We assume that the prior $\Pi_f$ on each emission distributions is a Dirichlet process. We compute the rate $\epsilon_n$ obtained with this prior when the true emission distributions have an exponential decay. 
In Section \ref{se:continuous}, the observations are assumed to live in $\mathbb{R}$ and the emission distributions are supposed to be absolutely continuous with respect to the Lebesgue measure. We consider a Dirichlet process mixture of Gaussian distributions  as a prior $\Pi_f$ on each emission density functions. We compute the rate $\epsilon_n$ obtained with this prior when the emission density functions belong to functional classes of $\beta$-Hölder types.

 We always assume that 
\begin{enumerate}
\labitem{(Q0)}{hyp:Q0} $\Pi_Q$ is absolutely continuous with respect to the Lebesgue measure on $\Delta^k_k$ with density $\pi_Q$, $\pi_Q(Q^*)>0$ and $\pi_Q(Q)=\pi_q(Q_{1,\cdot}) \dots \pi_q(Q_{k,\cdot})$, for all $Q  \in \Delta^k_k$ and where $Q_{i,\cdot}$ denotes the $i$-th row of $Q$.
\end{enumerate}
In Sections \ref{se:discrete} and \ref{se:continuous}, we consider three different priors $\Pi_Q$ on transition matrices which corresponds to three different decays of $\pi_Q$ near the boundary of $\Delta_k^k$:
\begin{enumerate}
\labitem{(Q1)}{hyp:pi_Q_exp} exponential tail: 
\[\pi_q(u_1,\dots , u_k) \lesssim \exp(-\alpha_1/u_1) \dots  \exp(-\alpha_k/u_k),\]
 for all $u\in \Delta_k$, for some positive constants $\alpha_{i},$ $1 \leq i \leq k$,
\labitem{(Q2)}{hyp:pi_Q_exp_exp} exponential of exponential tail: 
\[
\pi_q(u_1,\dots , u_k) \lesssim \exp(-\beta_1\exp(u_1^{-\alpha_1})) \dots   \exp(-\beta_k\exp(u_k^{-\alpha_k}))
,\]
for all $u\in \Delta_k$, for some positive constants $\alpha_{i}$ and $\beta_{i}$, $1 \leq i \leq k$,
\labitem{(Q3)}{hyp:pi_Q_tronque} truncated distribution: 
\[\Pi_Q(\Delta_k^k(\underline{q}))=1,\]
 for some positive $\underline{q}$.
\end{enumerate}
Note that Assumption \ref{hyp:pi_Q_tronque} implies Assumption \ref{hyp:pi_Q_exp_exp} which implies \ref{hyp:pi_Q_exp}.
 In \citep{GaRo12}, together with priors of type  \ref{hyp:pi_Q_exp},  more general priors  are also considered, since they assume
\[\pi_q(u_1,\dots , u_k) \lesssim u_1^{\alpha_1-1} \dots u_k^{\alpha_k-1},\]
 for all $u\in \Delta_k$, for some positive constants $\alpha_{i}$, $1 \leq i \leq k$. \cite{GaRo12} show that $\underline{q}_n$, in Assumption \eqref{hyp:qn} and Theorem \ref{th1}, is equal to a power of $1/n$ when the emission distributions belong to a parametric family. We do not consider this type of priors since they lead to deteriorated rates $\epsilon_n/ \underline{q}_n$. Under \ref{hyp:pi_Q_exp}, \citep{GaRo12} obtain $\underline{q}_n$ equal to a power of $1/(\log n)$ when the emission distributions belong to a parametric family. We obtain the same rate $\underline{q}_n$ in the case of discrete observations and emission distributions with exponential decay (more generally, it would be the case as soon as $\epsilon_n=n^{-1/2} \log(n)^t$ for some positive $t$). However, in the case of emission distributions absolutely continuous with respect to Lebesgue measure (Section \ref{se:continuous}), Assumption \eqref{hyp:qn} leads to a rate $\underline{q}_n$ at least polynomial in $1/n$ with priors satisfying \ref{hyp:pi_Q_exp}. 
While priors verifying \ref{hyp:pi_Q_exp_exp} or \ref{hyp:pi_Q_tronque} lead to a rate $\underline{q}_n$ equal to a power of $1/\log n$ as soon as $\epsilon_n$ is a power of $n$; and thus do not deteriorate the posterior concentration rate (up to $\log (n)$).  Note that Assumptions \ref{hyp:Q0} and \ref{hyp:pi_Q_tronque} are compatible if and only if $\underline{q}\leq \underline{q}^*$. Thus, the use of a prior verifying \ref{hyp:pi_Q_tronque} requires a knowledge of a lower bound $\underline{q}^*$ of $\min_{1 \leq i,j \leq k} Q^*_{i,j}$. 

 In \citep{Ve15}, posterior consistency is derived under Assumption \ref{hyp:pi_Q_tronque} while weaker conditions on $f^*$ and $\Pi_f^{(k)}$ (see  Assumptions (A0), (A1) and (A2) in \citep{Ve15}) compared to \eqref{hyp:KL}, \eqref{hyp:Fn} and \eqref{hyp:voisQ*}. Here, we manage to obtain posterior concentration rates under Assumptions \ref{hyp:pi_Q_exp} and \ref{hyp:pi_Q_exp_exp} which are weaker than Assumption \ref{hyp:pi_Q_tronque} because stronger conditions on $\Pi_f^{(k)}$ and $f^*$ are assumed.

\subsection{Discrete observations}\label{se:discrete}

In this section, we apply Theorem \ref{th1} to the case of discrete observations so that $\lambda$ is the count measure on $\mathbb{N}$. 
HMMs with discrete observations are used in different applications, as in \citep{BoZuHeCaLa13} for animal abundance estimation, in \citep{GaClRo13}for gene expression identification, or in \citep{LiJoWiCh14} for neural representation of spatial navigation, to cite a few. 

In the framework of discrete distribution estimation with i.i.d. observations, \citep{HaJiWe14} have proved that no rates can be obtained with the $L_1$ loss without constraint on the considered distributions. Moreover, they obtain a minimax rate proportional to $1/ \log n$ over the set $ \{f \in \mathcal{F}: ~ \sum_{ i \in \mathbb{N}} -f(i) \log (f(i)) \leq C\}$. Rates of convergence are more widely studied in the case of the $L_2$ norm, for instance with monotony constraint in \citep{JaWe09}, with log-concave constraint in \citep{BaJaRuPa13}, with convex constraint in \citep{DuHuKoRo13} and with envelope constraint in \citep{BoGa09}.

 In the non parametric Bayesian framework, the Dirichlet process is a very popular prior. Here, we consider a Dirichlet process $DP(G)$ on the emission distributions, with $G$ some finite positive measure on $\mathbb{N}$:
\[
\Pi_f =  DP(G).
\] \citep{CaDu11} propose other priors for discrete observations based on discretization of continuous mixtures of kernels and gives an overview of the priors used in the case of discrete observations.

We first verify Assumptions \eqref{hyp:KL}, \eqref{hyp:Fn} and \eqref{hyp:voisQ*}. As it can be seen from the proof (see Lemma \ref{prop:discret} in Appendix \ref{ap:discret}), we need a heavy tail condition on $G$:
\begin{enumerate}
\labitem{(P)}{eq:queue_G_pol} there exists positive constants $a\leq A$ and $\alpha\geq 2$ such that for all $1\leq j \leq k$ and for all $l\in \mathbb{N}$
\begin{equation*}
 a l^{-\alpha} \leq G(l)\leq A l^{-\alpha}.
\end{equation*}
\end{enumerate}

Here, we consider the following class of discrete distributions which is based on an envelope constraint:
\begin{multline}\label{eq:class_exp}
\mathcal{D}(m,c,K)=\Big\{f \in \mathcal{F}:  f(l) \leq d \exp(-c l^m), ~ \sum_{l\leq N} \frac{-\log(f(l))}{l} \lesssim N^{K},\\ \text{ for all } N \text{ large enough}\Big\},
\end{multline} 
where $c$, $K$ and $m$ are positive constants.
We also consider the following assumption linking the tails of the true emission distributions: 
\begin{enumerate}
\labitem{(I)}{hyp:dis_exp_fifj} there exists $\delta>0$ such that for all $N$ large enough and all $1\leq i,j \leq k$,
\begin{equation*}
 \sum_{l\geq N} f^*_i(l) \log^2\left(f^*_j(l)\right) \lesssim \exp(-N^m(c-\delta)).
\end{equation*}
\end{enumerate}

Under these assumptions, we obtain the following rates $\tilde{\epsilon}_n$ and $\epsilon_n$:

\begin{theorem}\label{th:epsilon_discret}
Assume there exist positive constants $c$, $K$ and $m$ such  that for all $1\leq j \leq k$, $f^*_j \in \mathcal{D}(m,c,K)$ and that Assumptions \ref{hyp:Q0}, \ref{hyp:dis_exp_fifj} and \ref{eq:queue_G_pol} hold.

Then Assumptions \eqref{hyp:KL}, \eqref{hyp:Fn} and \eqref{hyp:voisQ*}  hold with 
\[
\tilde{\epsilon}_n=\frac{1}{\sqrt{n}} (\log n)^{t_0} \quad \text{ and } \quad \epsilon_n=\frac{1}{\sqrt{n}} (\log n)^t,
\]
where $t>4t_0$ and $t_0 \geq 1/2 \max ( 1/m + 1, K/m)$.
\end{theorem}

Theorem \ref{th:epsilon_discret} leads to the following posterior concentration rates ($\epsilon_n/\underline{q}_n$) which are minimax (up to $\log n$):

\begin{corollary}\label{co:vit_dis}
Assume there exist positive constants $c$, $K$ and $m$ such  that for all $1\leq j \leq k$, $f^*_j \in \mathcal{D}(m,c,K)$ and that Assumptions \ref{hyp:Q0}, \ref{hyp:dis_exp_fifj} and \ref{eq:queue_G_pol} hold. Moreover suppose that $\Pi_Q$ satisfies
\begin{itemize}
\item   \ref{hyp:pi_Q_exp}, then the posterior concentrates with rate $\frac{1}{\sqrt{n}} (\log n)^{t+2t_0}$;
\item \ref{hyp:pi_Q_exp_exp}, then the posterior concentrates with rate $\frac{1}{\sqrt{n}} (\log n)^t $;
\item  \ref{hyp:pi_Q_tronque}, then the posterior concentrates with rate $\frac{1}{\sqrt{n}} (\log n)^t $;
\end{itemize}
with $t>4t_0$ and $t_0 \geq 1/2 \max ( 1/m + 1, K/m)$.
\end{corollary}


\subsection{Dirichlet process mixtures of Gaussian distributions--Adaptivity to Hölder function classes}\label{se:continuous}

Dirichlet process mixtures of Gaussian distributions are commonly used to model densities on $\mathbb{R}$ or $\mathbb{R}^d$. In particular, there exist efficient algorithms to sample from the posterior distribution in the i.i.d. framework.
In the translation HMM:
\begin{equation}\label{eq:HMM_translatee}
Y_t=m_{X_t} + \epsilon_t,
\end{equation}
where $\epsilon_t \overset{\text{i.i.d.}}{\sim} g \lambda$, $m_j \in \mathbb{R}$ and $X_t$ is a Markov chain with translation matrix $Q$;
\cite{YaPaRoHo11} use a Dirichlet process mixtures of Gaussian distributions on $g$.
In the context of i.i.d. observations
posterior concentration rates have been derived with such prior models in \cite{GhVa07dir}, \cite{KrRoVa10} and \cite{ShToGh13}. In the framework of HMMs, we propose to apply Theorem \ref{th1} when $\Pi_f$ is a Dirichlet process mixture of Gaussian distributions.

We assume that the reference measure $\lambda$ is the Lebesgue measure on $\mathbb{R}$. We also assume that the prior on $\mathcal{F}^k$ is a product of Dirichlet process mixture of Gaussian distributions:  
\begin{equation*}
\begin{split}
 &Y_t|X_t=j \sim f_j,  \quad f_j=\int \phi_\sigma(\cdot - \mu) dP_j(\mu),\\
 & P_j \overset{\text{i.i.d.}}{\sim} DP(G), \text{ for all }{~1 \leq j \leq k}, \quad \sigma \sim \pi_\sigma\lambda
,\end{split}\end{equation*}
where $\phi_\sigma$ is the Gaussian density function with variance $\sigma^2$ and mean zero, $DP(G)$ is the Dirichlet process with finite positive base measure $G$ and $\pi_\sigma$ is a distribution on $\mathbb{R}$.

We define the same functional classes as in \cite{KrRoVa10}:
\begin{multline}\label{hyp:B1}
\mathcal{P}(\beta,L,\gamma):=\Big\{ f \in \mathcal{F}: ~ \log f \text{ is locally } \beta\text{-Hölder with derivatives } 
l_j=(\log f)^{(j)}\\ \text{and } \lvert l_{\lfloor \beta \rfloor} (y) -  l_{\lfloor \beta \rfloor} (x) \rvert \leq r! L(y) \lvert y-x \rvert^{\beta - \lfloor \beta \rfloor},  \text { as soon as } \lvert x-y \rvert \leq \gamma \Big\}
\end{multline}
where $\beta>0$,  $L$ is a polynomial function, $\gamma >0$ and $\lfloor \beta \rfloor$ is the largest integer smaller than $\beta$.
We also consider the following tail conditions:
\begin{enumerate}
\labitem{(T1)}{hyp:B2} there exist positive constants  $M_0$, $\tau_0$, $\gamma_0$ such that for all $1 \leq i \leq k$ and all $y \in \mathbb{R}$
\[
f^*_i(y) \leq M_0 \exp(-\tau_0 \lvert y \rvert^{\gamma_0}),
\]
\labitem{(T2)}{hyp:B3}  for all $1 \leq i,j \leq k$ there exist constants $T_{i,j}$, $M_{i,j}$, $\tau_{i,j}$, $\gamma_{i,j}<\gamma_0$ such that
\[
f^*_i(y) \leq f^*_j(y) M_{i,j} \exp(\tau_{i,j} \lvert y \rvert^{\gamma_{i,j}}), \quad \lvert y \rvert \geq T_{i,j},
\]
\labitem{(T3)}{hyp:B4} for all $1 \leq i \leq k$, $f^*_i$ is positive and there exist $c_i>0$, $y_i^m < y_i^M$ such that $f^*_i$ is non decreasing on $(-\infty,y_i^m)$, non increasing on $(y_i^M, +\infty)$ and $f^*_i(y)\geq c_i$ for $y \in (y_i^m,y_i^M)$.
\end{enumerate}
Assumptions \ref{hyp:B2} and \ref{hyp:B4} are the same tail assumptions as those used in \cite{KrRoVa10}. The new Assumption \ref{hyp:B3} links the tail of each emission distributions.

We now describe the assumptions concerning the prior on the emission distributions:
\begin{enumerate}
\labitem{(G1)}{hyp:B7} $G([-y,y]^c) \lesssim  \exp(-C_1 y^{a_1})$ for all sufficiently large $y>0$, for some positive constant $a_1$,
\labitem{(S1)}{hyp:B8} $\Pi_\sigma\left( \sigma \leq x \right) \lesssim \exp(-C_2 x^{-a_2})$ for all sufficiently small $x>0$, for some positive constant $a_2$,
\labitem{(S2)}{hyp:B9} $\Pi_\sigma\left( \sigma > x \right) \lesssim  x^{-a_3}$ for all sufficiently large $x>0$, for some positive constant $a_3$,
\labitem{(S3)}{hyp:B10} there exists $a_6\leq 1$ such that
 \[\Pi_\sigma\left( x \leq \sigma \leq x(1+s) \right) \gtrsim x^{-a_4} s^{a_5} \exp(-C_3 x^{-a_6})\]
 for all $s \in (0,1)$, sufficiently small $x>0$, for some positive constants $a_4$, $a_5$ and $a_6$.
\end{enumerate}
The gamma and Gaussian distributions satisfy Assumption \ref{hyp:B7}.
The  inverse gamma distribution verifies \ref{hyp:B8}, \ref{hyp:B9} and \ref{hyp:B10}.

\begin{theorem}\label{th:vit_cont_vrai}Assume that there exist $\beta,L$ and $\gamma$ such that for all $1\leq j \leq k $, $f^*_j \in \mathcal{P}(\beta,L,\gamma)$ and Assumptions \ref{hyp:Q0}, \ref{hyp:B2}--\ref{hyp:B4}, \ref{hyp:B7} and \ref{hyp:B8}--\ref{hyp:B10} hold.

 Then Assumptions \eqref{hyp:KL}, \eqref{hyp:Fn} and \eqref{hyp:voisQ*} hold with 
\begin{equation}\label{eq:e_n_cont}
\tilde{\epsilon}_n=n^{-\frac{ \beta}{2 \beta +1}} \log(n)^{t_0} \quad \text{and} 
\quad \epsilon_n=n^{-\frac{ \beta}{2 \beta +1}} \log(n)^{t},
\end{equation}
where $t>t_0 \geq (2+2/\gamma_0 +1/\beta)/(1/\beta +2) $.
\end{theorem}

The proof of Theorem \ref{th:vit_cont_vrai} is given in Appendix \ref{se:pr_th_continuous}. Using Theorem \ref{th1} and \ref{th:vit_cont_vrai}, we directly deduce posterior rate of convergence under the Assumptions of Theorem \ref{th:vit_cont_vrai} and the different types of priors $\Pi_Q$. 

\begin{corollary}\label{co:vit_cont}
Assume that there exist $\beta,L$ and $\gamma$ such that for all $1\leq j \leq k $, $f^*_j \in \mathcal{P}(\beta,L,\gamma)$ and that Assumptions \ref{hyp:Q0}, \ref{hyp:B2}--\ref{hyp:B4}, \ref{hyp:B7} and \ref{hyp:B8}--\ref{hyp:B10} hold. Moreover suppose that $\Pi_Q$ satisfies
\begin{itemize}
\item   \ref{hyp:pi_Q_exp}, then the posterior concentrates with rate $n^{\frac{ -\beta+1}{2 \beta +1}} (\log n)^{3t} $;
\item \ref{hyp:pi_Q_exp_exp}, then the posterior concentrates with rate $n^{-\frac{ \beta}{2 \beta +1}} (\log n)^{t+1/(\max_{1 \leq i \leq k}\alpha_i)} $;
\item  \ref{hyp:pi_Q_tronque}, then the posterior concentrates with rate $n^{-\frac{ \beta}{2 \beta +1}} (\log n)^t$;
\end{itemize}
with $t>(2+2/\gamma_0 +1/\beta)/(1/\beta +2)$.
\end{corollary}

The minimax rate, with respect to $D_\ell$ in the HMM framework for emission density functions belonging to functional classes of $\beta$-Hölder type, is larger than $n^{- \beta/(2 \beta +1)}$. Indeed with a hidden Markov chain $(X_t,Y_t)$ distributed from a parameter $\theta=(Q,f)$ such that $Q_{i,j}=1/k$ and $f_i=f_1$ for all $1\leq i,j \leq k$,  the observations  $(Y_t)$ are i.i.d. from $f_1 \lambda$. 
Thus, priors satisfying \ref{hyp:Q0}, \ref{hyp:B2}--\ref{hyp:B4}, \ref{hyp:B7}, \ref{hyp:B8}--\ref{hyp:B10} and \ref{hyp:pi_Q_exp_exp} lead to minimax rates (up to $\log n$). As these priors do not depend on the regularity of the functional class considered, they ensure adaptive Bayesian density estimation in the framework of HMMs.


\section*{Acknowledgements}
I want to thank Elisabeth Gassiat and Judith Rousseau for their valuable comments. This work was partly supported by the grants  ANR Banhdits and Calibration.

\begin{appendices}

\section{Proof of Lemma \ref{lemma2v3} : control of the Kullback Leibler divergence between $\theta^*$ and $\theta$}\label{se:prKLv3}

First denote $\epsilon=\epsilon_n$.
Using Assumptions \eqref{hyp:q*} and \eqref{hyp:voisQ*2}, there exists $\underline{q}>0$ such that for all $1\leq i,j \leq k$,
$Q^*_{i,j}\geq \underline{q}$ and $Q_{i,j} \geq \underline{q}$, more precisely, $\underline{q}$ can be chosen equal to $\underline{q}^*/2$ as soon as $n$ is large enough.
Let $q^{\theta,Y_{1:t-1}}_t$ be the conditional density function of $Y_t$ given $Y_{1:t-1}$ with respect to $\lambda$:
\[
q^{\theta,Y_{1:t-1}}_t=\sum_{i=1}^k f_i(\cdot)Q^{\theta,Y_{1:t-1}}_{t,i},
\]
where $Q^{\mu,\theta,Y_{1:t-1}}_{t,i} = \mathbb{P}^{\theta}(X_t=i | Y_{1:t-1}, X_1\sim \mu)$
, where $t\geq 1$, $1 \leq i \leq k$,   $\mu \in \Delta_k$. When $\mu$ is not specified ($Q^{\theta,Y_{1:t-1}}_{t,i}$), the stationary initial probability distribution is considered.
Note that, when  $\mu \in \Delta_k(\underline{q})$, $\tilde{\theta} \in \Delta_k^k(\underline{q}) \times \mathcal{F}^k$,
\begin{equation}\label{eq:min_Q_sachant_obs}
Q^{\mu,\tilde{\theta},Y_{1:t-1}}_{t,i} = \frac{\sum_{j=1}^k Q^{\mu,\tilde{\theta},Y_{1:t-2}}_{t-1,j} Q_{j,i} f_j(Y_{t-1}) }{\sum_{\iota=1}^k Q^{\mu,\tilde{\theta},Y_{1:t-2}}_{t-1,\iota}  f_\iota(Y_{t-1}) } \geq \underline{q},
\end{equation}
 for all $1\leq i  \leq k$, $t \geq 1$.

\begin{equation}\label{eq:contKL}
\begin{split}
&KL(p_n^{\theta^*},p_n^\theta) \\
& =   \mathbb{E}^{\theta^*}\left(\sum_{t=1}^n  \int  q_t^{\theta^*,Y_{1:t-1}}(y) 
\log\left(\frac{q_t^{\theta^*,Y_{1:t-1}}(y)}{q_t^{\theta,Y_{1:t-1}}(y)}  \right) \lambda(dy)\right) \\
& = 
 \mathbb{E}^{\theta^*}\left(\sum_{t=1}^n  \int_S  q_t^{\theta^*,Y_{1:t-1}}(y) 
\log\left(\frac{q_t^{\theta^*,Y_{1:t-1}}(y)}{q_t^{\theta,Y_{1:t-1}}(y)}  \right) \lambda(dy)\right)  \\
&
\quad + \mathbb{E}^{\theta^*}\left(\sum_{t=1}^n  \int_{S^c}  q_t^{\theta^*,Y_{1:t-1}}(y) 
\log\left(\frac{q_t^{\theta^*,Y_{1:t-1}}(y)}{q_t^{\theta,Y_{1:t-1}}(y)}  \right) \lambda(dy)\right).
\end{split}
\end{equation}
Using Equation \eqref{eq:min_Q_sachant_obs}, for all $1 \leq l \leq k$,
\begin{equation}\label{eq:majoration_frac_q_theta}
\begin{split}
\frac{q_t^{\theta^*,Y_{1:t-1}}(y)}{q_t^{\theta,Y_{1:t-1}}(y)}
&=\frac{\sum_{i=1}^k f^*_i(y)Q^{\theta^*,Y_{1:t-1}}_{t,i}}{\sum_{i=1}^k f_i(y)Q^{\theta,Y_{1:t-1}}_{t,i}}
\leq \frac{k \max_{1 \leq j \leq k} f^*_j(y) }{ \underline{q}  f_l(y)} \\
&\leq \max_{1 \leq j \leq k} \frac{k  f^*_j(y) }{\underline{q}   f_j(y)}
\end{split}
\end{equation}
then Assumptions \eqref{hypa:maxloS} and \eqref{hypa:maxf*Sc} lead to
\begin{equation}\label{eq:contSc}
\mathbb{E}^{\theta^*}\left(\int_{S^c}  q_t^{\theta^*,Y_{1:t-1}}(y) 
\log\left(\frac{q_t^{\theta^*,Y_{1:t-1}}(y)}{q_t^{\theta,Y_{1:t-1}}(y)}  \right) \lambda(dy)\right)
\leq \left(\log\frac{k}{\underline{q}} + 1 \right) \epsilon^2.
\end{equation}
We now control the  expectancy of the third line of Equation \eqref{eq:contKL}
\begin{equation}\label{eq:contKLS}
\begin{split}
& \mathbb{E}^{\theta^*}\left(\sum_{t=1}^n  \int_S  q_t^{\theta^*,Y_{1:t-1}}(y) 
\log\left(\frac{q_t^{\theta^*,Y_{1:t-1}}(y)}{q_t^{\theta,Y_{1:t-1}}(y)}  \right) \lambda(dy)\right)\\
& \leq 
\mathbb{E}^{\theta^*}\left(\sum_{t=1}^n  \int_S  q_t^{\theta^*,Y_{1:t-1}}(y) 
\log\left(\frac{q_t^{\theta^*,Y_{1:t-1}}(y)}{\sum_{i=1}^k \tilde{f}_i(y)Q^{\theta,Y_{1:t-1}}_{t,i}}  \right) \lambda(dy)\right)\\
& \quad +
\mathbb{E}^{\theta^*}\left(\sum_{t=1}^n  \int_S  q_t^{\theta^*,Y_{1:t-1}}(y) 
\log\left(\frac{\sum_{i=1}^k \tilde{f}_i(y)Q^{\theta,Y_{1:t-1}}_{t,i}}{q_t^{\theta,Y_{1:t-1}}(y)}  \right) 
\lambda(dy)\right).
\end{split}
\end{equation}
We control the expectancy of the third line of Equation \eqref{eq:contKLS}, using 
\begin{equation*}
\begin{split}
&\frac{\sum_{i=1}^k \tilde{f}_i(y)Q^{\theta,Y_{1:t-1}}_{t,i}}{q_t^{\theta,Y_{1:t-1}}(y)} 
=\frac{\sum_{i=1}^k \tilde{f}_i(y)Q^{\theta,Y_{1:t-1}}_{t,i}}{\sum_{i=1}^k f_i(y)Q^{\theta,Y_{1:t-1}}_{t,i}} 
\leq \max_{1 \leq i \leq k}  \frac{\tilde{f}_i(y)}{f_i(y)};
\end{split}
\end{equation*}
by  Lemma \ref{le:astuce_frac_somme},
 and Assumption \eqref{hypa:maxloSc}, to obtain
 \begin{equation}\label{eq:kl_S}
 \mathbb{E}^{\theta^*}\left( \int_S  q_t^{\theta^*,Y_{1:t-1}}(y) 
\log\left(\frac{\sum_{i=1}^k \tilde{f}_i(y)Q^{\theta,Y_{1:t-1}}_{t,i}}{q_t^{\theta,Y_{1:t-1}}(y)}  \right) \lambda(dy)\right)
\leq   \epsilon^2.
 \end{equation}
We bound the expectancy of the second line of Equation  \eqref{eq:contKLS}, using the inequality recalled at the top of page
1234 of \citep{KrRoVa10},
\begin{equation}\label{eq:contSSc}
\begin{split}
&\mathbb{E}^{\theta^*}\left(  \int_S  q_t^{\theta^*,Y_{1:t-1}}(y) 
\log\left(\frac{q_t^{\theta^*,Y_{1:t-1}}(y)}{\sum_{i=1}^k \tilde{f}_i(y)Q^{\theta,Y_{1:t-1}}_{t,i}}  \right) \lambda(dy)\right)\\
&\leq \mathbb{E}^{\theta^*}\left( \int_S  \frac{\left( q_t^{\theta^*,Y_{1:t-1}}(y) - \sum_{i=1}^k \tilde{f}_i(y)Q^{\theta,Y_{1:t-1}}_{t,i} \right)^2}{\sum_{i=1}^k \tilde{f}_i(y)Q^{\theta,Y_{1:t-1}}_{t,i}} 
 \lambda(dy)\right) \\
 &
\quad +  \mathbb{E}^{\theta^*}\left(  \int_{S^c}  \sum_{i=1}^k \tilde{f}_i(y)Q^{\theta,Y_{1:t-1}}_{t,i} - q_t^{\theta^* ,Y_{1:t-1}}(y)  \lambda(dy) \right) .
\end{split}
\end{equation}
The expectancy of the second line of Equation  \eqref{eq:contSSc} is controlled as follows
\begin{equation}\label{eq:contSQf}
\begin{split}
&\mathbb{E}^{\theta^*}\left( \int_S  \frac{\left( q_t^{\theta^*,Y_{1:t-1}}(y) - \sum_{i=1}^k \tilde{f}_i(y)Q^{\theta,Y_{1:t-1}}_{t,i} \right)^2}{\sum_{i=1}^k \tilde{f}_i(y)Q^{\theta,Y_{1:t-1}}_{t,i}} 
 \lambda(dy)\right) \\
&\leq 2\mathbb{E}^{\theta^*}\left(  \frac{k \max_{1 \leq i \leq k }\left( Q_{t,i}^{\theta^*,Y_{1:t-1}} - Q_{t,i}^{\theta,Y_{1:t-1}} \right)^2}{\underline{q}} \right)
 +  2\mathbb{E}^{\theta^*}\left( \int_S  \max_{1 \leq i \leq k} \frac{\left( f^*_i(y) - \tilde{f}_i(y) \right)^2}{\underline{q} \tilde{f}_i(y)} 
 \lambda(dy)\right) \\
&\leq \left( \frac{16k(1+2k)}{\underline{q}^4} +1\right)\frac{2}{\underline{q}} \epsilon^2 +\frac{16 k\rho^{2(t-1)}}{\underline{q}}
\end{split}
\end{equation}
using Lemma \ref{lesumQ} and then Assumption \eqref{hypa:qui2} and Lemma \ref{lesumQ}.
The expectancy of the third line of Equation \eqref{eq:contSSc} is controlled thanks to Assumption \eqref{hypa:maxSc}:
\begin{equation}\label{eq:contfSc}
\begin{split}
&\mathbb{E}^{\theta^*}\left(  \int_{S^c} \sum_{i=1}^k \tilde{f}_i(y)Q^{\theta,Y_{1:t-1}}_{t,i}  - q_t^{\theta^* ,Y_{1:t-1}}(y)  \lambda(dy) \right)  \\
&\leq \mathbb{E}^{\theta^*}\left(  \int_{S^c}  \sum_{i=1}^k \tilde{f}_i(y)Q^{\theta,Y_{1:t-1}}_{t,i}   \lambda(dy) \right)  \\
&\leq \max_{1 \leq i \leq k } \int_{S^c}  \tilde{f}_i(y) \lambda(dy) \leq \epsilon^2
.
\end{split}
\end{equation}
We conclude the proof by combining Equations \eqref{eq:contKL}, \eqref{eq:contSc}, \eqref{eq:contKLS}, \eqref{eq:kl_S}, \eqref{eq:contSSc}, \eqref{eq:contSQf} and \eqref{eq:contfSc}.\hfill$\qed$

\begin{lemma}\label{le:astuce_frac_somme}
For all $a_i, b_i, c_i, d_i \geq 0$, $1 \leq i \leq k$,
\begin{equation*}
\frac{\sum_{1\leq i \leq k} a_i b_i}{\sum_{1\leq j \leq k} c_j d_j }
= \frac{\sum_{1\leq i \leq k} a_i/c_i * b_i/d_i *  c_i d_i}{\sum_{1\leq j \leq k} c_j d_j } 
\leq \max_{1\leq i \leq k} \frac{a_i}{c_i} \max_{1\leq j \leq k} \frac{b_j}{d_j}.
\end{equation*}
\end{lemma}


\section{Proof of Lemma \ref{lemma1} with technical lemmas: control of  $Var^{\theta^*}(L_n^{\theta^*}-L_n^\theta)$}\label{se:preuve_KL2}

\subsection{Proof of Lemma \ref{lemma1}} 
 
First denote $\epsilon=\tilde{\epsilon}_n/\sqrt{u_n}$.
 Using Assumptions \eqref{hyp:q*} and \eqref{hyp:voisinage_Q*}, there exists $\underline{q}>0$ such that for all $1\leq i,j \leq k$,
$Q^*_{i,j}\geq \underline{q}$ and $Q_{i,j} \geq \underline{q}$, more precisely, $\underline{q}$ can be chosen equal to $\underline{q}^*/2$ as soon as $n$ is large enough.
Let  $Var$ be the variance of $L_n^{\theta^*}-L_n^\theta$ :
\[
Var:=\mathbb{E}^{\theta^*}\left[
\left(\log \frac{p_n^{\theta^*}(Y_{1:n})}{p_n^{\theta}(Y_{1:n})}
       -\mathbb{E}^{\theta^*}\left(\log \frac{p_n^{\theta^*}(Y_{1:n})}{p_n^{\theta}(Y_{1:n})}\right)\right)^2\right]. 
       \]
Denoting $Z_t=\log \left(\frac{P^{\theta^*}(Y_t | Y_{1:t-1})}{P^{\theta}(Y_t | Y_{1:t-1})} \right)$, then
\[
Var=\mathbb{E}^{\theta^*}\left( \left(  \sum_{t=1}^n Z_t - \mathbb{E}^{\theta^*}\left(\sum_{t=1}^n Z_t \right)  \right)^2  \right).
\]
We want to bound  $Var$ by $Cn\epsilon^{(2-\alpha)/2}$, for any $\alpha>0$.
In this purpose, we split the sum in two parts: 
\begin{equation}\label{var}
Var \leq 2 \underbrace{\mathbb{E}^{\theta^*}\left[
\left( \sum_{t=1}^n \left(Z_t -\mathbb{E}^{\theta^*}\left(Z_t |Y_{1:t-1} \right)\right)\right)^2\right]}_{=S_1}
          +2 \underbrace{\mathbb{E}^{\theta^*}\left[
\left( \sum_{t=1}^n \left(\mathbb{E}^{\theta^*}\left(Z_t |Y_{1:t-1} \right)  -  \mathbb{E}^{\theta^*} (Z_t) \right)\right)^2\right]}_{=S_2},
       \end{equation}
$S_1$ is the expectancy of the square of a sum of martingale increments, for which the covariances are zero so that only $n$ terms remain. $S_2$ is further controlled using the exponential forgetting of Markov chain.
First, we control $S_1$:
\begin{equation}\label{s1}
\begin{split}
S_1 & = \sum_{t=1}^n \mathbb{E}^{\theta^*}\left(\left(Z_t -\mathbb{E}^{\theta^*}\left(Z_t |Y_{1:t-1} \right)\right)^2\right) \\
& \quad + 2\sum_{1\leq r<t\leq n}\mathbb{E}^{\theta^*}\bigg( \big(Z_r-\mathbb{E}^{\theta^*}(Z_r|Y_{1:r-1})\big)
~ \mathbb{E}^{\theta^*}\big(Z_t - \mathbb{E}^{\theta^*}(Z_t |Y_{1:t-1} ) |Y_{1:t-1}\big) \bigg) \\
& \leq \sum_{t=1}^n \mathbb{E}^{\theta^*}\left(Z_t^2\right)
\end{split}
\end{equation}
using the  convexity of the square function and Jensen's inequality so that
\begin{equation}\label{eq:convexsquare}
\mathbb{E}^{\theta^*}(\mathbb{E}^{\theta^*}(Z_t|Y_{1:t-1})^2) \leq \mathbb{E}^{\theta^*}(Z_t^2).
\end{equation}
As to $S_2$:
\begin{equation}\label{s2}
\begin{split}
S_2 & =  \sum_{t=1}^n \mathbb{E}^{\theta^*}\left(\left(\mathbb{E}^{\theta^*}\left(Z_t |Y_{1:t-1} \right)   -  \mathbb{E}^{\theta^*}\left(Z_t\right) \right)^2\right) \\
& \hspace{1cm} + 2 \sum_{1\leq r<t\leq n}\mathbb{E}^{\theta^*}\bigg( \big(\mathbb{E}^{\theta^*}(Z_r|Y_{1:r-1})-  \mathbb{E}^{\theta^*}(Z_r)\big)
~ \mathbb{E}^{\theta^*}\big(\mathbb{E}^{\theta^*}(Z_t |Y_{1:t-1} ) - \mathbb{E}^{\theta^*}(Z_t) |Y_{1:r-1}\big) \bigg) \\
& \leq  \sum_{t=1}^n \mathbb{E}^{\theta^*}(Z_t^2) + 
2\sum_{1\leq r<t\leq n} \sqrt{\mathbb{E}^{\theta^*}(Z_r^2)} ~  
\sqrt{\mathbb{E}^{\theta^*}\big(\lvert \mathbb{E}^{\theta^*}(Z_t |Y_{1:r-1} ) - \mathbb{E}^{\theta^*}(Z_t)\rvert^2\big)},
\end{split}
\end{equation}
using Equation \eqref{eq:convexsquare} and Cauchy-Schwarz inequality to bound the second term. 

Combining \eqref{var},  \eqref{s1} et \eqref{s2}, we obtain
\begin{equation}\label{va}
Var \leq 4 \sum_{t=1}^n \mathbb{E}^{\theta^*}\left(Z_t^2\right)
+ 4 \sum_{1\leq r<t\leq n} \sqrt{\mathbb{E}^{\theta^*}(Z_r^2)} ~  
\sqrt{\mathbb{E}^{\theta^*}\big(\lvert \mathbb{E}^{\theta^*}(Z_t |Y_{1:r-1} ) - \mathbb{E}^{\theta^*}(Z_t)\rvert^2\big)}.
\end{equation}

Then using Lemmas \ref{lemmaEZt2} and \ref{le:esp_diffzt},
\begin{multline}
Var  \leq 4 \left( \frac{16}{\underline{q}(1- \rho^2)} + Cn \epsilon^2 \right)
 + 16 \rho^{-\frac{5\alpha}{4}} \left( 2\epsilon +\frac{10}{\underline{q}} \right)^{\alpha/2} 
  \\ \sum_{1\leq r<t\leq n} \left(\frac{16 \rho^{2(r-1)}}{\underline{q}^2} + \tilde{C} \epsilon^2\right)^{1-\alpha/4}
  \rho^{\frac{\alpha}{4}(t-r)}.
\end{multline}
Since,
\begin{equation}
\begin{split}
 \sum_{1\leq r<t\leq n} &\left(\frac{16 \rho^{2(r-1)}}{\underline{q}^2} + \tilde{C} \epsilon^2\right)^{1-\alpha/4}
  \rho^{\frac{\alpha}{4}(t-r)} 
 \leq 2   \frac{\rho^{\alpha/4}}{1-\rho^{\alpha/4}}   \left( \frac{16}{\underline{q}^2(1- \rho)}   + n \tilde{C} \epsilon^{2-\alpha/2} \right)
\end{split}
\end{equation}
therefore there exists a constant $C_{KL^2}>0$ only depending on k and $\underline{q}^*(=2 \underline{q})$ such that
\begin{equation}
\begin{split}
Var 
\leq C_{KL^2}\frac{n}{\alpha} \left( \frac{\tilde{\epsilon}_n}{\sqrt{u_n}}\right)^{2-\alpha}.
\end{split}
\end{equation}


\subsection{Lemma \ref{lemmaEZt2}: control of $\mathbb{E}^{\theta^*}\left(Z_t^2\right)$}

\begin{lemma}\label{lemmaEZt2}
For all $\theta, \theta^* \in \Delta_k^k(\underline{q}) \times \mathcal{F}^k$
\begin{equation}\label{eq:esp_zt2}
\begin{split}
& \mathbb{E}^{\theta^*}(Z_t^2)\\
& \leq \frac{16 \rho^{2(t-1)}}{\underline{q}^2}
+ 2 \max_{1\leq i \leq k} \int f^*_i(y) \max_{1 \leq j \leq k} \log^2\left( \frac{f^*_j(y)}{f_j(y)}\right) \lambda(dy)\\
& \quad + 
  \frac{32 ~ \lVert Q- Q^*\rVert^2}{\underline{q}^4(1-\rho)^2} 
+ \frac{32}{\underline{q}^4(1-\rho)^2} 
\int  \min\left( 
 \sum_{1 \leq i \leq k} \frac{  \lvert f^*_i(y) - f_i(y)  \rvert^2}{f^*_i(y)} ,\underline{q}^2 k^2 \sum_{1 \leq j \leq k} f^*_j(y)\right) \lambda(dy),
 \end{split}
\end{equation}
with $\rho=\frac{1-k \underline{q}}{1-(k-1)\underline{q}} \leq 1-\underline{q}$.

If moreover Assumptions \eqref{hypa:log2}, \eqref{hypa:maxf*Sc}, \eqref{hypa:qui2S} and \eqref{hyp:voisinage_Q*} hold, then
\begin{equation}\label{eq:esp_zt22}
\mathbb{E}^{\theta^*}(Z_t^2) \leq \frac{16 \rho^{2(t-1)}}{\underline{q}^2} + \tilde{C}\frac{\tilde{\epsilon}_n^2}{u_n}
\end{equation}
where $\tilde{C} \leq 33(1+2k)/\underline{q}^6$.
\end{lemma}

\begin{proof}[Proof of Lemma \ref{lemmaEZt2}]
Let $Q^{\mu,\theta,Y_{1:t-1}}_{t,i} = \mathbb{P}^{\theta}(X_t=i | Y_{1:t-1}, X_1\sim \mu)$
, where $t\geq 1$, $1 \leq i \leq k$,   $\mu \in \Delta_k$ and when $\mu$ is not specified, the stationary initial probability distribution is considered. Then the conditional density function of $Y_t$ given $Y_{1:t-1}$ with respect to $\lambda$ is:
\[
\sum_{i=1}^k f_i(\cdot)Q^{\theta,Y_{1:t-1}}_{t,i}; 
\]
so that
\begin{equation}\label{Z}
\begin{split}
\mathbb{E}^{\theta^*}\left(Z_t^2\right)
 & =\mathbb{E}^{\theta^*}\left(\left(\log\frac{\sum_{i=1}^{k} f^*_i(Y_t)Q^{\theta^*, Y_{1:t-1}}_{t,i}}{\sum_{j=1}^{k}  f_j(Y_t)Q^{\theta, Y_{1:t-1}}_{t,j}}
\right)^2\right)  \\
 & \leq 2 \mathbb{E}^{\theta^*}\left(\max_{1 \leq j \leq k} \log^2\left(\frac{f^*_j(Y_t)}{f_j(Y_t)}\right)\right)
+ \frac{2}{\underline{q}^2} 
\mathbb{E}^{\theta^*}\left( \left( 
\sum_{j=1}^k 
\lvert Q^{\theta^*, Y_{1:t-1}}_{t,j}- Q^{\theta, Y_{1:t-1}}_{t,j}\rvert 
\right)^2 \right)
\end{split}
\end{equation}
using Equation \eqref{eq:min_Q_sachant_obs}
and Lemma \ref{le:astuce_frac_somme}.

Combining Equation \eqref{Z} and Lemma \ref{lesumQ} (Equation \eqref{eq_lemmeeq1}), we obtain Equation \eqref{eq:esp_zt2}.
Moreover, using Assumption \eqref{hypa:log2}, 
\begin{equation}\label{eq:emaxlog2}
\mathbb{E}^{\theta^*}\left(\max_{1 \leq j \leq k} \log^2\left(\frac{f^*_j(Y_t)}{f_j(Y_t)}\right)\right)
\leq \epsilon^2.
\end{equation}
Finally, combining Equations \eqref{Z}, \eqref{eq:emaxlog2} and Lemma \ref{lesumQ} (Equation \eqref{eq_lemmeeq2}), we obtain Equation \eqref{eq:esp_zt22}.
\end{proof}

\begin{lemma}\label{lesumQ}
For all $\theta, \theta^* \in \Delta_k^k(\underline{q}) \times \mathcal{F}^k$ and $\mu, \mu^* \in \Delta_k(\underline{q})$,
\begin{equation}\label{eq_lemmeeq1}
\begin{split}
&\mathbb{E}^{\theta^*}\left( \left( 
\sum_{j=1}^k 
\lvert Q^{\theta^*, Y_{1:t-1}}_{t,j}- Q^{\theta, Y_{1:t-1}}_{t,j}\rvert 
\right)^2 \right)\\
&\leq 
8 \rho^{2(t-1)}
+ 
  \frac{16 ~ \lVert Q- Q^*\rVert^2}{\underline{q}^2(1-\rho)^2} \\
& \quad 
+ \frac{16}{\underline{q}^2(1-\rho)^2} 
\int  \min\left( 
 \sum_{1 \leq i \leq k} \frac{  \lvert f^*_i(y) - f_i(y)  \rvert^2}{f^*_i(y)} ,\underline{q}^2 k^2 \sum_{1 \leq j \leq k} f^*_j(y)\right) \lambda(dy),
 \end{split}
\end{equation}
with $\rho=\frac{1-k \underline{q}}{1-(k-1)\underline{q}} \leq 1-\underline{q}$.

If moreover Assumptions \eqref{hypa:log2}, \eqref{hypa:maxf*Sc}, \eqref{hypa:qui2S}  and \eqref{hyp:voisinage_Q*} hold, then
\begin{equation}\label{eq_lemmeeq2}
\mathbb{E}^{\theta^*}\left( \left( 
\sum_{j=1}^k 
\lvert Q^{\theta^*, Y_{1:t-1}}_{t,j}- Q^{\theta, Y_{1:t-1}}_{t,j}\rvert 
\right)^2 \right)
\leq 8 \rho^{2(t-1)} + C' \frac{\tilde{\epsilon}_n^2}{u_n}
\end{equation}
where $C'\leq 16(1+2k)/\underline{q}^4$.
\end{lemma}

\begin{proof}[Proof of Lemma \ref{lesumQ}]

We first control  $\sum_{i=1}^k \lvert  Q^{\theta^*,Y_{1:t-1}}_{t,i} - Q^{\mu^*,\theta,Y_{1:t-1}}_{t,i} \rvert $.
For this purpose, we are going to use  a modified version of Proposition 1 of \cite{DoMa01}.
By Proposition 1 of \cite{DoMa01} and for all $\theta, \theta^*$ in $\Delta_k^k(\underline{q}) \times \mathcal{F}^k$ we can control 
the $L_1$-norm between two conditional probabilities
of the state $t$ when the initial probabilities are equal.

\begin{equation}\label{3}
\begin{split}
 \sum_{i=1}^k & \lvert  Q^{\theta^*,Y_{1:t-1}}_{t,i} - Q^{\mu^*,\theta,Y_{1:t-1}}_{t,i} \rvert \\
& \leq  \sum_{j=1}^k \left\lvert Q^{(Q_{2,\cdot}^{\theta^*,Y_1},Q,f),Y_{2:t-1}}_{t-1,j}-
Q^{(Q_{2,\cdot}^{\theta^*,Y_1},Q^*,f^*),Y_{2:t-1}}_{t-1,j}\right\rvert  + \frac{1}{2}  \left( \frac{1-k \underline{q}}{1-(k-1)\underline{q}} \right) ^{t-1}* \\
 &  \qquad  \sum_{u=1}^k \Bigg\lvert  \frac{\sum_{j=1}^k (A^{Y_{t-1},\theta} \dots A^{Y_2,\theta})_{j,u}
(A^{Y_1,\theta}\mu^*)_u}{\sum_{v=1}^k (A^{Y_{t-1},\theta} \dots A^{Y_2,\theta}A^{Y_1,\theta}\mu^*)_v} -  \frac{\sum_{j=1}^k (A^{Y_{t-1},\theta} \dots A^{Y_2,\theta})_{j,u}
(A^{Y_1,\theta^*}\mu^*)_u}{\sum_{v=1}^k (A^{Y_{t-1},\theta} \dots A^{Y_2,\theta}A^{Y_1,\theta^*}\mu^*)_v} \Bigg\rvert
\end{split}
 \end{equation}
with $\left(A^{Y,\theta}_{i,j}\right)_{1 \leq i,j\leq k}= \left(Q_{j,i} f_j(Y)\right)_{1 \leq i,j\leq k}$.
And 
\begin{equation}\label{4}
\begin{split}
 \sum_{u=1}^k & \Bigg\lvert  \frac{\sum_{j=1}^k (A^{Y_{t-1},\theta} \dots A^{Y_2,\theta})_{j,u}
(A^{Y_1,\theta}\mu^*)_u}{\sum_{v=1}^k (A^{Y_{t-1},\theta} \dots A^{Y_2,\theta}A^{Y_1,\theta}\mu^*)_v}  -  \frac{\sum_{j=1}^k (A^{Y_{t-1},\theta} \dots A^{Y_2,\theta})_{j,u}
(A^{Y_1,\theta^*}\mu^*)_u}{\sum_{v=1}^k (A^{Y_{t-1},\theta} \dots A^{Y_2,\theta}A^{Y_1,\theta^*}\mu^*)_v} \Bigg\rvert\\
& = \sum_{u=1}^k \Bigg\lvert  \frac{\sum_{j=1}^k (A^{Y_{t-1},\theta} \dots A^{Y_2,\theta})_{j,u}
(A^{Y_1,\theta^*}\mu^* - A^{Y_1,\theta}\mu^*)_u}{\sum_{v=1}^k (A^{Y_{t-1},\theta} \dots A^{Y_2,\theta}A^{Y_1,\theta^*}\mu^*)_v}\\
& \quad -\frac{\sum_{w=1}^k \sum_{i=1}^k (A^{Y_{t-1},\theta} \dots A^{Y_2,\theta})_{i,w}
(A^{Y_1,\theta^*}\mu^* - A^{Y_1,\theta}\mu^*)_w}{\sum_{v=1}^k (A^{Y_{t-1},\theta} \dots A^{Y_2,\theta}A^{Y_1,\theta^*}\mu^*)_v} \\
& \qquad \frac{\sum_{j=1}^k (A^{Y_{t-1},\theta} \dots A^{Y_2,\theta})_{j,u}
(A^{Y_1,\theta}\mu^*)_u}{\sum_{v=1}^k (A^{Y_{t-1},\theta} \dots A^{Y_2,\theta}A^{Y_1,\theta}\mu^*)_v}
\Bigg\rvert\\
& \leq 2  \sum_{u=1}^k \min\bigg(  \left\lvert  \frac{\sum_{j=1}^k (A^{Y_{t-1},\theta} \dots A^{Y_2,\theta})_{j,u}
(A^{Y_1,\theta^*}\mu^* - A^{Y_1,\theta}\mu^*)_u}{\sum_{v=1}^k (A^{Y_{t-1},\theta} 
\dots A^{Y_2,\theta}A^{Y_1,\theta^*}\mu^*)_v}\right\rvert  ,1\bigg) \\
& \leq 2\min \left( \max_{1\leq u\leq k} \frac{ \sum_{i=1}^k \mu^*_i  \lvert Q_{i,u} f_i(Y_1) - Q^*_{i,u} f_i^*(Y_1)\rvert}{
 \sum_{i=1}^k \mu^*_i Q^*_{i,u} f_i^*(Y_1)} , k \right)  \\
& \leq 2  \left( \frac{\lVert Q -Q^* \rVert}{\underline{q}}  +  
\min\left(\frac{(1-(k-1)\underline{q})}{\underline{q}}
\frac{\sum_{i=1}^k \lvert f_i^*(Y_1) - f_i(Y_1)  \rvert \mu^*_i }{\sum_{j=1}^k f_j^*(Y_1) \mu^*_j},k \right) \right), 
\end{split}
\end{equation}
using Lemma \ref{le:astuce_frac_somme}.
Combining Equations (\ref{3}) and (\ref{4}),

\begin{equation}\label{5}
 \begin{split}
\sum_{i=1}^k & \lvert  Q^{\theta^*,Y_{1:t-1}}_{t,i} - Q^{\mu^*,\theta,Y_{1:t-1}}_{t,i} \rvert \\
 & \leq  \sum_{j=1}^k \left\lvert Q^{(Q_{2,\cdot}^{\theta^*,Y_1},Q,f),Y_{2:t-1}}_{t-1,j}-
Q^{(Q_{2,\cdot}^{\theta^*,Y_1},Q^*,f^*),Y_{2:t-1}}_{t-1,j}\right\rvert 
 +   \left( \frac{1-k \underline{q}}{1-(k-1)\underline{q}} \right) ^{t-1}* \\
& \qquad \left( \frac{\lVert Q -Q^* \rVert}{\underline{q}}  +  
\min\left(\frac{(1-(k-1)\underline{q})}{\underline{q}}
\frac{\sum_{i=1}^k \lvert f_i^*(Y_1) - f_i(Y_1)  \rvert \mu^*_i }{\sum_{j=1}^k f_j^*(Y_1) \mu^*_j},k\right) \right).
\end{split}
\end{equation}

By repeating the arguments of Equation (\ref{4}), we show that 
\begin{equation*}
\begin{split}
\sum_{i=1}^k & \lvert  Q^{(Q^{\theta^*,Y_{1:t-2}}_{t-1},Q^*,f^*),Y_{1:t-1}}_{2,i} - Q^{(Q^{\theta^*,Y_{1:t-2}}_{t-1},Q,f),Y_{1:t-1}}_{2,i} \rvert \\
&=  \sum_{i=1}^k \left\lvert   \frac{\sum_{i=1}^k A^{Y_{t-1},\theta^*}_{j,i}Q^{\theta^*,Y_{1:t-2}}_{t-1,i}}{\sum_{i,u=1}^k A^{Y_{t-1},\theta^*}_{u,i}Q^{\theta^*,Y_{1:t-2}}_{t-1,i}} 
- \frac{\sum_{i=1}^k A^{Y_{t-1},\theta}_{j,i}Q^{\theta^*,Y_{1:t-2}}_{t-1,i}}{\sum_{i,u=1}^k A^{Y_{t-1},\theta}_{u,i}Q^{\theta^*,Y_{1:t-2}}_{t-1,i}}    \right\rvert \\
& \leq 2 \left( \frac{\lVert Q -Q^* \rVert}{\underline{q}}  +  
\min\left(\frac{(1-(k-1)\underline{q})}{\underline{q}}
\frac{\sum_{i=1}^k \lvert f_i^*(Y_{t-1}) - f_i(Y_{t-1})  \rvert Q^{\theta^*,Y_{1:t-2}}_{t-1,i} }{\sum_{j=1}^k f_j^*(Y_{t-1}) Q^{\theta^*,Y_{1:t-2}}_{t-1,j}} ,k \right) \right).
\end{split}
\end{equation*}

By induction on (\ref{5}),

\begin{equation}\label{6}
 \begin{split}
\sum_{i=1}^k & |Q^{\theta^*,Y_{1:t-1}}_{t,i} - Q^{\mu^*,\theta,Y_{1:t-1}}_{t,i}| \\
&\leq 2 \left( \frac{\lVert Q -Q^* \rVert}{\underline{q}}  +  
\min\left(\frac{(1-(k-1)\underline{q})}{\underline{q}}
\frac{\sum_{i=1}^k \lvert f_i^*(Y_{t-1}) - f_i(Y_{t-1})  \rvert Q^{\theta^*,Y_{1:t-2}}_{t-1,i} }{\sum_{j=1}^k f_j^*(Y_{t-1}) Q^{\theta^*,Y_{1:t-2}}_{t-1,j}},k \right)  \right) \\ 
& \quad + \sum_{u=3}^{t}   \left( \frac{1-k \underline{q}}{1-(k-1)\underline{q}} \right)^{u-1}\\
& \quad \left(  \frac{\lVert Q -Q^* \rVert}{\underline{q}}  +  
\min\left(\frac{(1-(k-1)\underline{q})}{\underline{q}} \frac{\sum_i \lvert f_i^*(Y_{t-u+1}) - 
 f_i(Y_{t-u+1})  \rvert Q^{\theta^*,Y_{1:t-u}}_{t-u+1,i}
 }{\sum_j f_j^*(Y_{t-u+1}) Q^{\theta^*,Y_{1:t-u}}_{t-u+1,j}},k \right)  \right).
 \end{split}
\end{equation}

 Using Corollary 1 of \cite{Do04},  
 we can control the $\ell_1$-norm between two conditional probabilities of the state $t$
 for the same parameter $\theta$ but different initial probabilities:
 for all $\theta \in \Delta_k^k(\underline{q}) \times \mathcal{F}^k$, $\mu, \tilde{\mu} \in \Delta_k$ and for all $y_{1:l-1} \in \{y ~ : ~ \exists i,~ f^*_i(y)>0  \}^{l-1}$
\begin{equation}\label{7}
\sum_{i=1}^k \left\lvert  Q^{\mu,\theta,y_{1:l-1}}_{i,l} - Q^{\tilde{\mu},\theta,y_{1:l-1}}_{i,l} \right\rvert
\leq 2 \rho^{l-1}
.\end{equation}
 
Combining Equations (\ref{6}) and (\ref{7}), we obtain
 
\begin{equation}\label{8}
 \begin{split}
\sum_{i=1}^k & |Q^{\mu^*,\theta^*,Y_{1:t-1}}_{t,i} - Q^{\mu,\theta,Y_{1:t-1}}_{t,i}| \\
&\leq 2 \left( \frac{\lVert Q -Q^* \rVert}{\underline{q}}  +  
\min\left(\frac{1}{\underline{q}}
\frac{\sum_{i=1}^k \lvert f_i^*(Y_{t-1}) - f_i(Y_{t-1})  \rvert Q^{\theta^*,Y_{1:t-2}}_{t-1,i} }{\sum_{j=1}^k f_j^*(Y_{t-1}) Q^{\theta^*,Y_{1:t-2}}_{t-1,j}} ,k\right)  \right)\\ 
& \quad +2 \sum_{u=3}^{t}   \left( \frac{1-k \underline{q}}{1-(k-1)\underline{q}} \right)^{u-1}\\
& \quad \underbrace{\left(  \frac{\lVert Q -Q^* \rVert}{\underline{q}}  +  
\min\left(\frac{1}{\underline{q}} \frac{\sum_i \lvert f_i^*(Y_{t-u+1}) - 
 f_i(Y_{t-u+1})  \rvert Q^{\theta^*,Y_{1:t-u}}_{t-u+1,i}
 }{\sum_j f_j^*(Y_{t-u+1}) Q^{\theta^*,Y_{1:t-u}}_{t-u+1,j}}  ,k\right)\right)}_{=\Delta_{t-u+1}}\\
 & \quad + 4 \left( \underbrace{\frac{1-k \underline{q}}{1-(k-1)\underline{q}}}_{=\rho} \right)^{t-1} \\
 & \leq \frac{2}{\rho} \sum_{u=1}^{t-1} \rho^{u} \Delta_{t-u} + 4\rho^{t-1}.
\end{split}
\end{equation}

Then
\begin{equation}\label{var_carre}
 \begin{split}
\bigg(\sum_{i=1}^k  \lvert Q^{\mu^*,\theta^*,Y_{1:t-1}}_{t,i} - Q^{\mu,\theta,Y_{1:t-1}}_{t,i}\rvert \bigg)^2 
& \leq \frac{8}{\rho^2} \left(\sum_{u=1}^{t-1} \rho^u \right)
\left(\sum_{u=1}^{t-1} \rho^u \Delta_{t-u}^2\right) +8 \rho^{2(t-1)},
\end{split}
\end{equation}
using Cauchy-Schwarz inequality.
Moreover, using Lemma \ref{le:astuce_frac_somme},
\begin{equation}\label{ED}
\begin{split}
&\mathbb{E}^{\theta^*}(\Delta_{t-u}^2)\\
& \leq 2 \frac{\lVert Q- Q^*\rVert^2}{\underline{q}^2} +  \frac{2}{\underline{q}^2}  \mathbb{E}^{\theta^*}\left( \min\left( \left(
\frac{\sum_i \lvert f_i^*(Y_{t-u}) - 
 f_i(Y_{t-u})  \rvert Q^{\theta^*,Y_{1:t-u-1}}_{t-u,i}
 }{\sum_j f_j^*(Y_{t-u}) Q^{\theta^*,Y_{1:t-u-1}}_{t-u,j}} 
 \right)^2,(\underline{q}k)^2\right) \right)\\
 & \leq 2 \frac{\lVert Q- Q^*\rVert^2}{\underline{q}^2} \\
 &  \hspace{1cm} +
 \frac{2}{\underline{q}^2}
 E\left( \int \sum_{1\leq \iota \leq k} Q_{t-u,\iota}^{\theta^*,Y_{1:t-u-1}}f_\iota(y)
 \min\left( \left(
\frac{\sum_i \lvert f_i^*(y) - 
 f_i(y)  \rvert Q^{\theta^*,Y_{1:t-u-1}}_{t-u,i}
 }{\sum_j f_j^*(y) Q^{\theta^*,Y_{1:t-u-1}}_{t-u,j}} 
 \right)^2,(\underline{q}k)^2\right)   \lambda(dy) \right) \\
 & \leq 2 \frac{\lVert Q- Q^*\rVert^2}{\underline{q}^2} + 
 \frac{2}{\underline{q}^2}
 \int  \min\left( 
 \sum_{1 \leq i \leq k} \frac{  \lvert f^*_i(y) - f_i(y)  \rvert^2}{f^*_i(y)} ,\underline{q}^2 k^2 \sum_{1 \leq j \leq k} f^*_j(y)\right) \lambda(dy).
%
\end{split}
\end{equation}
Combining Equations \eqref{var_carre} and \eqref{ED}, we obtain Equation \eqref{eq_lemmeeq1} which implies Equation \eqref{eq_lemmeeq2} under Assumptions  \eqref{hypa:log2}, \eqref{hypa:maxf*Sc}, \eqref{hypa:qui2S} and \eqref{hyp:voisinage_Q*}. This concludes the proof of Lemma \ref{lesumQ}.

\end{proof}
%

\subsection{Lemma \ref{le:esp_diffzt}: control of $\mathbb{E}^{\theta^*}\big(\lvert \mathbb{E}^{\theta^*}(Z_t |Y_{1:r-1} ) - \mathbb{E}^{\theta^*}(Z_t)\rvert^2\big)$}

In the following lemma we show that   $\mathbb{E}^{\theta^*}\left(\big(\lvert \mathbb{E}^{\theta^*}(Z_t |Y_{1:r-1} )- \mathbb{E}^{\theta^*}(Z_t)\rvert\big)^2 \right)$ geometrically decreases to $0$ when $t$ tends to $+\infty$, using the exponential forgetting of the Markov chain.

\begin{lemma}\label{le:esp_diffzt}
For all $\theta, \theta^* \in \Delta_k^k(\underline{q}) \times \mathcal{F}^k$ and $\alpha \in (0,2)$,
\begin{equation}\label{eq:finoubli}
\begin{split}
& \mathbb{E}^{\theta^*} \big(\lvert \mathbb{E}^{\theta^*}(Z_t |Y_{1:r-1} ) - \mathbb{E}^{\theta^*}(Z_t)\rvert^2\big) \\
& \leq 8 \mathbb{E}^{\theta^*}( Z_t^2 )^{\frac{2-\alpha}{2}}\rho^{-\frac{5\alpha}{2}}  \rho^{\frac{\alpha}{2}(t-r)}  \bigg(2\max_{1\leq j \leq k} \int f^*_j(y) \max_{1 \leq i \leq k} \left\lvert \log\frac{f^*_i(y)}{f_i(y)} \right\rvert \lambda(dy)
+ \frac{10}{\underline{q}}\bigg)^\alpha .
\end{split}
\end{equation}
If moreover Assumption \eqref{hypa:log2} holds then 
\begin{equation}\label{eq:finoubli2}
\begin{split}
& \mathbb{E}^{\theta^*} \big(\lvert \mathbb{E}^{\theta^*}(Z_t |Y_{1:r-1} ) - \mathbb{E}^{\theta^*}(Z_t)\rvert^2\big)  \leq 8 \mathbb{E}^{\theta^*}( Z_t^2 )^{\frac{2-\alpha}{2}}\rho^{-\frac{5\alpha}{2}}  \rho^{\frac{\alpha}{2}(t-r)}  \bigg(2 \frac{\tilde{\epsilon}_n}{\sqrt{u_n}}
+ \frac{10}{\underline{q}}\bigg)^\alpha .
\end{split}
\end{equation}
\end{lemma}

\begin{proof}[Proof of Lemma \ref{le:esp_diffzt}]
Denote
\[
L_t= (X_t, Y_t, Q^{\theta, Y_{1:t-1}}_{t, \cdot},Q^{\theta^*, Y_{1:t-1}}_{t, \cdot}) 
\]
 for all $t \in \mathbb{N}$, then $(L_t)_{t\in \mathbb{N}}$ is the extended Markov chain with transition kernel
 $\Pi_{\theta}$ more precisely described in \citep{DoMa01} at page 384.
Let
\[
h \colon \begin{cases}
\{1, \dots, k\} \times \{y : \exists 1\leq j \leq k, f^*_j(y)>0\} \times \{\mu \in \Delta_k : \mu_i>\underline{q}  ~ \forall i \}^2
\longrightarrow \mathbb{R}\\
l=(x,y,\mu, \mu^*)\longmapsto h(l)=\log\left( \frac{\sum_{i=1}^k \mu^*_i f^*_i(y)}{\sum_{i=1}^k \mu_i f_i(y)} \right)
\end{cases}
\]
then 
\[
h(L_t)=Z_t=\log\left(\frac{p^{\theta^*}(Y_t | Y_{1:t-1})}{p^{\theta}(Y_t | Y_{1:t-1})} \right)
\]
and for all $r\leq t$ and $0< \alpha < 2$,
\begin{equation}\label{EC2}
\begin{split}
\mathbb{E}^{\theta^*} & \big(\lvert \mathbb{E}^{\theta^*}(Z_t |Y_{1:r-1} ) - \mathbb{E}^{\theta^*}(Z_t)\rvert^2\big)\\
& =  \mathbb{E}^{\theta^*}  \big(\lvert \mathbb{E}^{\theta^*}(Z_t |Y_{1:r-1} ) - \mathbb{E}^{\theta^*}(Z_t)\rvert ^{2-\alpha} 
\lvert \mathbb{E}^{\theta^*}(Z_t |Y_{1:r-1} ) - \mathbb{E}^{\theta^*}(Z_t)\rvert ^{\alpha}\big) \\
& \leq 2^{2-\alpha} \mathbb{E}^{\theta^*}\big( 
\left(\max(\mathbb{E}^{\theta^*}(\lvert Z_t \rvert),  \mathbb{E}^{\theta^*}(\lvert Z_t\rvert |Y_{1:r-1} )  )\right)^{2- \alpha}
\lvert \mathbb{E}^{\theta^*}(h(L_t) |Y_{1:r-1} ) - \mathbb{E}^{\theta^*}(h(L_t))\rvert ^\alpha \big).\\
\end{split}
\end{equation}
The following term is geometrically decreasing, using Lemma \ref{lem:DoMa}
\begin{equation}\label{eq:oubli_exp_chaine_etendue}
\begin{split}
&\lvert \mathbb{E}^{\theta^*}(h(L_t) |Y_{1:r-1} ) - \mathbb{E}^{\theta^*}(h(L_t))\rvert \\
&\leq  \int \int \left\lvert \int h(l_t) \Pi_{\theta}^{t-r}(l_r,dl_t)-\int h(l_t) \Pi_{\theta}^{t-r}(\tilde{l}_r,dl_t) \right\rvert P^{\theta}(dl_r |Y_{1:r-1}) 
P^{\theta}(d\tilde{l}_r).
\end{split}
\end{equation}
More precisely, using Equation \eqref{eq:oubli_exp_chaine_etendue}  and Lemma \ref{lem:DoMa} with $m=\lfloor\frac{t-r+1}{2} \rfloor$ and $u=t-r$, we obtain
\begin{equation}\label{eq:esp_h_Lt}
\lvert \mathbb{E}^{\theta^*}(h(L_t) |Y_{1:r-1} ) - \mathbb{E}^{\theta^*}(h(L_t))\rvert 
\leq 
\rho^{\frac{t-r}{2} -\frac{5}{2}} \bigg(2 \max_{1\leq j \leq k} \int f^*_j(y) \max_{1 \leq i \leq k} \left\lvert \log\frac{f^*_i(y)}{f_i(y)} \right\rvert \lambda(dy)
+ \frac{10}{\underline{q}}\bigg).
\end{equation}
Therefore using Equations \eqref{EC2} and \eqref{eq:esp_h_Lt}, 
\begin{equation}\label{eq:fin}
\begin{split}
\mathbb{E}^{\theta^*} & \big(\lvert \mathbb{E}^{\theta^*}(Z_t |Y_{1:r-1} ) - \mathbb{E}^{\theta^*}(Z_t)\rvert^2\big)\\
& \leq 2^{2-\alpha} \mathbb{E}^{\theta^*} \bigg(
   \max(\mathbb{E}^{\theta^*}(\left( \lvert Z_t \rvert),  \mathbb{E}^{\theta^*}(\lvert Z_t\rvert |Y_{1:r-1} ))  \right)^{2- \alpha} \bigg)  \\
& \qquad  \rho^{-\frac{5\alpha}{2}}  \rho^{\frac{\alpha}{2}(t-r)}  \bigg(2\max_{1\leq j \leq k} \int f^*_j(y) \max_{1 \leq i \leq k} \left\lvert \log\frac{f^*_i(y)}{f_i(y)} \right\rvert \lambda(dy)
+ \frac{10}{\underline{q}}\bigg)^\alpha 
\end{split}
\end{equation}
By convexity of the square function and concavity of $x \mapsto x^{\frac{2- \alpha}{2}}$, with $0< \alpha < 2$,
\begin{equation}\label{eq:concav}
\begin{split}
& \mathbb{E}^{\theta^*} \bigg(
   \max(\mathbb{E}^{\theta^*}(\lvert Z_t \rvert),  \mathbb{E}^{\theta^*}(\lvert Z_t\rvert |Y_{1:r-1} )  )^{2- \alpha} \bigg)\\
& \leq  \mathbb{E}^{\theta^*} \bigg(
   \max(\mathbb{E}^{\theta^*}( Z_t^2)^{1/2},  \mathbb{E}^{\theta^*}( Z_t^2 |Y_{1:r-1} ) ^{1/2} )^{2- \alpha} \bigg) \\
& \leq \mathbb{E}^{\theta^*}(Z_t^2)^{\frac{2- \alpha}{2}} + \mathbb{E}^{\theta^*}( \mathbb{E}^{\theta^*}( Z_t^2 |Y_{1:r-1} ) ^{\frac{2- \alpha}{2}})\\
& \leq 2 \mathbb{E}^{\theta^*}(Z_t^2)^{\frac{2- \alpha}{2}}.
\end{split}\end{equation}
Combining Equations \eqref{eq:fin} and \eqref{eq:concav}, we get Equation \eqref{eq:finoubli}.
Besides, using Assumption \eqref{hypa:log2} and Cauchy–Schwarz inequality ,
\begin{equation}\label{eq:maxlog}
\max_{1\leq j \leq k} \int f^*_j(y) \max_{1 \leq i \leq k} \left\lvert \log\frac{f^*_i(y)}{f_i(y)} \right\rvert \lambda(dy)
\leq  \epsilon
\end{equation}
so that  Equation \eqref{eq:finoubli2} holds.

\end{proof}

 Lemma \ref{lem:DoMa}  is an improved version of Proposition 2 of \citep{DoMa01}.
\begin{lemma}\label{lem:DoMa}
For all integers $u>0$, $m<u$, for all $z, \tilde{z} \in \{1, \dots, k\} \times \{y : \exists 1\leq j \leq k, f^*_j(y)>0\} \times \{\mu \in \Delta_k : \mu_i>\underline{q}  ~ \forall i \} \times \{\mu \in \Delta_k : \mu_i>\underline{q}  ~ \forall i \}$ and for all $\theta, \theta^* \in \Delta_k^k(\underline{q})\times \mathcal{F}^k$,
\begin{equation}
\begin{split}
&\left\lvert \int h(l) \Pi_{\theta}^{u}(z,dl)-\int h(l) \Pi_{\theta}^{u}(\tilde{z},dl)\right\rvert\\
& \leq \frac{4}{\underline{q}} \rho^{u-1} +  \frac{4}{\underline{q}} \rho^{m} + 2 \rho^{m-2} \bigg(\max_{1\leq j \leq k} \int f^*_j(y) \max_{1 \leq i \leq k} \left\lvert \log\frac{f^*_i(y)}{f_i(y)} \right\rvert \lambda(dy)
+ \log\left(\frac{1}{\underline{q}}\right)\bigg)\\
\end{split}
\end{equation}
\end{lemma}

\begin{proof}[Proof of Lemma \ref{lem:DoMa}]

This lemma is an improved version of Proposition 2 of \cite{DoMa01}. We improve the result by defining $h$ on 
\[
 \mathcal{Z}= \{1, \dots, k\} \times \{y : \exists 1\leq j \leq k, f^*_j(y)>0\} \times \{\mu \in \Delta_k : \mu_i>\underline{q}  ~ \forall i \} \times \{\mu \in \Delta_k : \mu_i>\underline{q}  ~ \forall i \}
 \]
  and using that if  $z\in \mathcal{Z}$ then $\Pi_{\theta^*}(z, \mathcal{Z} )=1$.
Then we obtain
\begin{equation}\label{eq:bornelip}
\text{lip}(h,x,y)=\frac{1}{\underline{q}}
\end{equation}
since  for all $(x,y,\mu, \mu^*), (x,y,\tilde{\mu}, \tilde{\mu}^*) \in \mathcal{Z}$
\begin{equation*}
\begin{split}
 \lvert h(x,y,\mu, \mu^*) - h(x,y,\tilde{\mu}, \tilde{\mu}^*) \rvert 
  & =\left\lvert  \log\left(  \frac{\sum_{i=1}^k \mu^*_i f^*_i(y)}{\sum_{i=1}^k \mu_i f_i(y)} \right)  
   - \log\left(  \frac{\sum_{i=1}^k \tilde{\mu}^*_i f^*_i(y)}{\sum_{i=1}^k \tilde{\mu}_i f_i(y)} \right)       \right\rvert \\
& =\left\lvert  \log\left(  \frac{\sum_{i=1}^k \mu^*_i f^*_i(y)}{\sum_{i=1}^k \tilde{\mu}^*_i f^*_i(y)} \right)  
   - \log\left(  \frac{\sum_{i=1}^k \mu_i f_i(y)}{\sum_{i=1}^k \tilde{\mu}_i f_i(y)} \right)       \right\rvert \\
 &  \leq \max_{ 1\leq i \leq k}  \left\lvert  \log\left( \frac{\mu^*_i}{\tilde{\mu}^*_i} \right)  \right\rvert   
   + \max_{ 1\leq i \leq k}   \left\lvert \log\left( \frac{\mu_i}{\tilde{\mu}_i} \right)  \right\rvert\\
& \leq \frac{1}{\underline{q}}\left( \sum_{i=1}^k \lvert \mu^*_i - \tilde{\mu}^*_i \rvert   + \sum_{i=1}^k \lvert \mu_i - \tilde{\mu}_i \rvert \right)
 \end{split}
 \end{equation*}
 using . Moreover
\begin{equation}\label{eq:bornek}
k(h,x,y)=  \log\frac{1}{\underline{q}} + \max_{1 \leq i \leq k} \left\lvert  \log\left(  \frac{  f^*_i(y)}{f_i(y)} \right)  \right\rvert 
\end{equation}
 because, using Lemma \ref{le:astuce_frac_somme},
\begin{equation*}
\begin{split}
 \lvert h(x,y,\mu, \mu^*) \rvert 
   &=\left\lvert  \log\left(  \frac{\sum_{i=1}^k \mu^*_i f^*_i(y)}{\sum_{i=1}^k \mu_i f_i(y)} \right)  \right\rvert \\
& \leq  \max_{1 \leq i \leq k} \left\lvert  \log\left(  \frac{ \mu^*_i f^*_i(y)}{ \mu_i f_i(y)} \right)  \right\rvert \\
 &\leq  \max_{1 \leq i \leq k} \left\lvert  \log\left(  \frac{ \mu^*_i }{ \mu_i } \right)  \right\rvert + 
\max_{1 \leq i \leq k} \left\lvert  \log\left(  \frac{  f^*_i(y)}{f_i(y)} \right)  \right\rvert  \\
& \leq \log\frac{1}{\underline{q}} + \max_{1 \leq i \leq k} \left\lvert  \log\left(  \frac{  f^*_i(y)}{f_i(y)} \right)  \right\rvert.
 \end{split}
 \end{equation*}
Moreover instead of using Proposition 1 of \citep{DoMa01} we use Corollary 1 of \cite{Do04} so that for all $\theta \in \Delta_k^k(\underline{q}) \times \mathcal{F}^k$, $\mu, \tilde{\mu} \in \Delta_k(\underline{q})$ and for all $y_{1:l-1} \in \{y ~ : ~ \exists i,~ f^*_i(y)>0  \}^{l-1}$
\begin{equation}\label{eq:controleoubli}
\sum_{i=1}^k \left\lvert  Q^{\mu,\theta,y_{1:l-1}}_{i,l} - Q^{\tilde{\mu},\theta,y_{1:l-1}}_{i,l} \right\rvert
\leq 2 \rho^{l-1}
.\end{equation}

Then, let
$z=(x,y,\mu, \mu^*) \in \mathcal{Z}$ and $\tilde{z}=(\tilde{x}, \tilde{y}, \tilde{\mu}, \tilde{\mu}^*) \in \mathcal{Z}$ using the proof of Proposition 1 of \citep{DoMa01}:

\begin{equation}\label{ABCD}
\begin{split}
&\lvert \int h(l) \Pi^{u}(z,dl)-\int h(l) \Pi^{u}(\tilde{z},dl)\rvert\\
&\leq \underbrace{\left\lvert \int h(l) \Pi^{u}((x,y,\mu, \mu^*),dl)-\int h(l) \Pi^{u}((x, \tilde{y}, \tilde{\mu}, \tilde{\mu}^*),dl)\right\rvert}_{=A}\\
& \quad +\underbrace{\left\lvert \int h(l) \Pi^{u}(x, \tilde{y}, \tilde{\mu}, \tilde{\mu}^*),dl)-\int h(l) \Pi^{u}((\tilde{x}, \tilde{y}, \tilde{\mu}, \tilde{\mu}^*),dl)\right\rvert}_{=B}
\end{split}
\end{equation}
where
\begin{equation}\label{eq:majA}
\begin{split}
&\lvert A \rvert
 \leq \bigg\lvert\sum_{x_{2:u+1}=1}^k \int \text{lip}(h,x_{u+1},y_{u+1}) \\
& \left( \sum_{i=1}^k \lvert Q_{u,i}^{Q_1^{\mu,\theta,y},\theta,y_{2:u}} -
 Q_{u,i}^{Q_1^{\tilde{\mu},\theta,\tilde{y}},\theta,y_{2:u}} \rvert 
+ \sum_{i=1}^k \lvert 
Q_{u,i}^{Q_1^{\mu^*,\theta^*,y},\theta^*,y_{2:u}}
 -  Q_{u,i}^{Q_1^{\tilde{\mu}^*,\theta^*,\tilde{y}},\theta^*,y_{2:u}}
 \rvert  \right) \\
& \qquad Q^*_{x,x_2}\dots Q^*_{x_{u},x_{u+1}} f^*_{x_2}(y_2) \dots f^*_{x_{u+1}}(y_{u+1})\lambda(dy_{2}) \dots \lambda(dy_{u+1}) \bigg\rvert  
\end{split}
\end{equation}
and for any $ 1\leq m \leq u$,
\begin{equation}\label{eq:majB}
\begin{split}
&\lvert  B \rvert
\leq  \sum_{x_{2:u+1}=1}^k \int \text{lip}(h,x_{u+1},y_{u+1})
\bigg( \sum_{i=1}^k \lvert  Q_{m+1,i}^{Q_{u+1-m}^{\tilde{\mu}, \theta, \tilde{y},y_{2:u-m}}, \theta, y_{u-m+1:u}}  - Q_{m+1,i}^{\nu, \theta, y_{u-m+1:u}} \rvert \\
& \qquad +  \sum_{i=1}^k \lvert Q_{m+1,i}^{Q_{u+1-m}^{\tilde{\mu}, \theta^*, \tilde{y},y_{2:u-m}}, \theta^*, y_{u-m+1:u}}
-Q_{m+1,i}^{\nu^*, \theta^*, y_{u-m+1:u}} \rvert \bigg)\\
& \qquad \big\lvert Q^*_{x,x_2}- Q^*_{\tilde{x},x_2}\big\rvert Q^*_{x_2,x_3}\dots Q^*_{x_{u},x_{u+1}} f^*_{x_2}(y_2) \dots f^*_{x_{u+1}}(y_{u+1})\lambda(dy_{2:u+1}) \\
& \qquad + \sum_{x_{m:u+1}=1}^k \int k(h,x_{u+1}, y_{u+1})  \\
& \qquad \lvert {Q^*}^{m-1}_{x,x_m} - {Q^*}^{m-1}_{\tilde{x},x_m}  \rvert  Q^*_{x_m,x_{m+1}}\dots Q^*_{x_{u},x_{u+1}} f^*_{x_m}(y_m) \dots f^*_{x_{u+1}}(y_{u+1})\lambda(dy_{2:u+1})
.\end{split}
\end{equation}

Combining Equations \eqref{eq:bornelip},  \eqref{eq:controleoubli} and \eqref{eq:majA},
\begin{equation}\label{eq:majo_A}
\lvert A \rvert \leq   \frac{4}{\underline{q}} \rho^{u-1} 
\end{equation}
and using Equations  \eqref{eq:bornelip}, \eqref{eq:bornek}, \eqref{eq:controleoubli} and \eqref{eq:majB}
\begin{equation}\label{eq:majoB}
 \lvert B \rvert  \leq \frac{4}{\underline{q}} \rho^{m} + 2 \rho^{m-2} \bigg(\max_{1\leq j \leq k} \int f^*_j(y) \max_{1 \leq i \leq k} \left\lvert \log\frac{f^*_i(y)}{f_i(y)} \right\rvert \lambda(dy)
+ \log\left(\frac{1}{\underline{q}}\right)\bigg)  
\end{equation}
therefore using Equations \eqref{ABCD}, \eqref{eq:majo_A} and \eqref{eq:majoB} we obtain

\begin{equation}
\begin{split}
&\left\lvert \int h(l) \Pi^{u}(z,dl)-\int h(l) \Pi^{u}(\tilde{z},dl)\right\rvert\\
& \leq \frac{4}{\underline{q}} \rho^{u-1} +  \frac{4}{\underline{q}} \rho^{m} + 2 \rho^{m-2} \bigg(\max_{1\leq j \leq k} \int f^*_j(y) \max_{1 \leq i \leq k} \left\lvert \log\frac{f^*_i(y)}{f_i(y)} \right\rvert \lambda(dy)
+ \log\left(\frac{1}{\underline{q}}\right)\bigg).\\
\end{split}
\end{equation}

\end{proof}


\section{Proof of Theorem \ref{th:epsilon_discret} (discrete observations)}\label{ap:discret}

Assumption \eqref{hyp:Fn} will be checked using 
Proposition 2 of \citep{ShToGh13} that we recall here.
\begin{lemma}\label{le:b_c_discret}[Proposition 2 of \citep{ShToGh13}]
Let $H$ be a positive integer, $A$ and $\epsilon$ be positive, denote
\[\mathcal{H}_{H,A,\epsilon}=
  \left\{ f = \sum_{h=1}^{+\infty} \pi_h \delta_{z_h} ~ : ~
   \sum_{h>H} \pi_h  < \epsilon,
   z_h \in [0,A], h \leq H   
    ~  \right\}^k.
    \]
    Then
    \begin{equation}
    DP( G)\left( \mathcal{H}_{H,A,\epsilon} ^c\right)
    \leq k H G((A, +\infty)) + k \left( \frac{e G(\mathbb{N})}{H} \log \frac{1}{\epsilon} \right)^H
    \end{equation}
    \begin{equation}
    N( 4 l\epsilon , \mathcal{H}_{H,A,\epsilon} , d ) \lesssim  (A+1)^{kH} \epsilon^{-kH}.
    \end{equation}
\end{lemma}

We now give the proof of Theorem \ref{th:epsilon_discret}.

\begin{proof}[Proof of Theorem \ref{th:epsilon_discret}]
We first prove Assumption \eqref{hyp:KL} with $\tilde{f}_j=f_j $, for all $1 \leq j \leq k$ using Lemma \ref{prop:discret}  with 
$\epsilon = \tilde{\epsilon}_n^2/u_n$, 
$L=L_n=\left( -\log \left( \tilde{\epsilon}_n^2/(u_n \log \log n)\right)/(c-\delta) \right)^{1/m}$
 and $S_L=\{1, \dots , L\}$.
Using that $f^*_i\in \mathcal{D}(m,c,K)$ for all $1 \leq i \leq k$ and Assumption  \ref{eq:queue_G_pol}, we get
\begin{equation*}
\begin{split}
\sum_{l>L_n} \frac{f^*_i(l)}{(G(l))^2}
&\lesssim 
\sum_{l>L_n} \exp(-cl^m) l^{2\alpha}
\lesssim
\sum_{l>L_n} \exp(-(c-\delta)l^m) l^{m-1} \\
&\lesssim
\int_{L_n}^\infty \exp(-(c-\delta)x^m) x^{m-1} \lambda(dx)
\lesssim
 \exp(-(c-\delta)L_n^m)
 \lesssim
 \tilde{\epsilon}^{2}_n
 \end{split}
\end{equation*}
which proves Equation \eqref{eq:disf_g}. Equation \eqref{eq:disf} is proved similarly.
Equation \eqref{eq:dislog2} follows using Assumption \ref{hyp:dis_exp_fifj}.
Then, we can apply Lemma \ref{prop:discret} so that
\begin{equation}\label{eq:minoration_KL_discret}
\begin{split}
&DP(G)^{\otimes k} 
\bigg(f_j ~ : \forall 1 \leq i,j \leq k  ~ \sum_{l=1}^{+\infty} f^*_i(l)\log^2\left(\frac{f^*_j(l)}{f_j(l)} \right) \leq \frac{\epsilon_n^2}{u_n}, \\
 &
  \sum_{l=1}^{L_n}  \frac{\left(f^*_j(l)-f_j(l) \right)^2}{f^*_j(l)} \leq \frac{\epsilon_n^2}{u_n},    
   \sum_{l>L_n}f_j(l) \leq \frac{\epsilon_n^2}{u_n},   
  \sum_{l=1}^{L_n} \frac{\left( f^*_j(l) - f_j(l) \right)^2}{f_j(l)}  \leq \frac{\epsilon_n^2}{u_n} \bigg) \\
   & \gtrsim
   \prod_{j=1}^k \Bigg(  \left( \frac{\epsilon_n^2}{4 u_n}  \right)^{(L_n-1+G(\mathbb{N}))/2}     	f^*_j(l^*)^{L_n-2} \left( \frac{1}{3} \right)^{L_n}
   \prod_{l=1}^{L_n} G(l) f^*_j(l)^{G(l)}   \Bigg) .
 \end{split}
  \end{equation}
Moreover using that $f^*_i\in \mathcal{D}(m,c,K)$ for all $1 \leq i \leq k$ and Equation \ref{eq:queue_G_pol}, we obtain
\begin{equation}\label{eq:mino_prod_G_L}
\log \prod_{l=1}^{L_n} G(l)
\gtrsim   -L_n \log (L_n),
\end{equation}
\begin{equation}\label{eq:mino_prod_f_G}
\log \left(\prod_{l=1}^{L_n} f^*_j(l)^{G(l)}\right)
\gtrsim
 \sum_{l=1}^{L_n} G(l) \log f^*_j(l)
 \gtrsim L_n^{K}.
\end{equation} 
 Combining Equations \eqref{eq:minoration_KL_discret} with $l^*=\text{argmax}_{l}(\min_{1 \leq j \leq k} f^*_j(l))$, \eqref{eq:mino_prod_G_L} and \eqref{eq:mino_prod_f_G}, Assumption \eqref{hyp:KL} of Theorem \ref{th1} is true if 
 \[
  \left(- \log\left(\tilde{\epsilon}_n\right) \right)^{\max(1/m+1,K/m)} \lesssim n \tilde{\epsilon}_n^2.
 \]
  Then we choose
 \[
 \tilde{\epsilon}_n = \frac{1}{\sqrt{n}} (\log n)^{t_0} 
 \]
  with $2 t_0 > \max(1/m+1,K/m)$ and Assumption \eqref{hyp:KL} holds.
 
 Using Assumption \ref{hyp:Q0}, for $\tilde{\epsilon}_n$ small enough,
\begin{equation}\label{eq:mino_piQ}
\Pi_Q\left(\left\{Q : \lVert Q - Q^* \rVert \leq \frac{\tilde{\epsilon}_n}{\sqrt{u_n}} \right\}\right)
\geq
\frac{\pi(Q^*)}{2} \lambda\left(\left\{Q : \lVert Q - Q^* \rVert \leq \frac{\tilde{\epsilon}_n}{\sqrt{u_n}} \right\}\right)
\gtrsim \left(\frac{\tilde{\epsilon}_n}{\sqrt{u_n}}\right)^{k(k-1)}
\end{equation} 
so that Assumption \eqref{hyp:voisQ*} holds.
 
 Using Lemma \ref{le:b_c_discret} with  $A=\exp((\log n)^{2t_0+t/2})$, 
 $H=(n \epsilon_n^2)/((\log n)^{2t_0+t})$ and $\epsilon_n=(\log n)^t /\sqrt{n}$,
 \begin{equation}
 \begin{split}
 &\Pi(\mathcal{F}_n^c) \\
  &
 \lesssim
	(\log n)^{t-2t_0} \exp \left( -(\alpha+1) (\log n)^{2t_0+t/2} \right)
	  + \exp \left(  -(t-2t_0-1)(\log n)^{t-2t_0} (\log \log n)^2 \right)
  \\
&=o(\exp(-C' (\log n)^{2t_0}))
 =o(\exp(-C' n\tilde{\epsilon}_n^2)),
 \end{split}
 \end{equation}
  if $t>4t_0$. Moreover
 \begin{equation}
 \log \left(N\left( \frac{\epsilon_n}{12}, \mathcal{F}_n, D_\ell \right) \right)
 \lesssim (\log n)^{3t/2} + (\log n)^{t-2t_0 +1}
 \lesssim n\epsilon_n^2
 \end{equation}
 so that Assumption \eqref{hyp:Fn} holds.
This concludes the proof of Theorem \ref{th:epsilon_discret}.

\end{proof}

\begin{lemma}\label{prop:discret}
Let $S_L$ be a subset of $\{1, \dots, L\}$.
If
\begin{equation}\label{eq:dislog2}
\max_{1 \leq i,j \leq k} \sum_{S_L^c} f^*_i(l) \log^2(f^*_j(l)) \leq \frac{\epsilon}{8}
\end{equation}
 and 
\begin{equation}\label{eq:disf_g}
\max_{1 \leq i \leq k} \sum_{S_L^c} \frac{f^*_i(l)}{(G(l))^2} \leq \frac{\epsilon}{384 \log^2(2) k(G(\mathbb{N}))^2 } ,
\end{equation}
and if there exists $\delta>0$ such that
\begin{equation}\label{eq:disf}
\max_{1 \leq i \leq k} \sum_{l\in S_L^c} f^*_i(l) \leq \epsilon^{1+\delta} ,
\end{equation}
then,  for all $l^* \in S_L$ and all $\epsilon>0$ small enough,
\begin{equation*}
\begin{split}
 P_G&:=DP(G)^{\otimes k} 
\bigg(f_j ~ : 
\forall 1 \leq i,j \leq k  ~ \sum_{l=1}^{+\infty} f^*_i(l)\log^2\left(\frac{f^*_j(l)}{f_j(l)} \right) \leq \epsilon, \\
 & \hspace{1cm} \sum_{l\in S_L}  \frac{\left(f^*_j(l)-f_j(l) \right)^2}{f^*_j(l)} \leq \epsilon,    
   \sum_{l \in S_L^c} f_j(l) \leq \epsilon,   
  \sum_{l \in S_L} \frac{\left( f^*_j(l) - f_j(l) \right)^2}{f_j(l)}  \leq \epsilon \bigg) \\
&\gtrsim \prod_{j=1}^k \Bigg(  \left( \sqrt{\frac{\epsilon}{4}}  \right)^{L-1+G(\mathbb{N})}   
 \left( \frac{1}{3} \right)^{L}
  f^*_j(l^*)^{L-2} 
\left[ \prod_{l\in S_L} G(l) f^*_j(l)^{G(l)} \right]
  \Bigg).
\end{split}
\end{equation*}

\end{lemma}

\begin{proof}[Proof of Lemma \ref{prop:discret}]

Note that if for all 
$ l \in S_L$
 and for all 
 $1\leq j \leq k$,
\[\left(1-\sqrt{\frac{\epsilon}{4}}\right)f^*_j(l) \leq f_j(l) \leq \left(1+\sqrt{\frac{\epsilon}{4}}\right) f^*_j(l)\]
then for all
$ l \in S_L$,
\[
\log^2 \left(\frac{f^*_j(l)}{f_j(l)}\right) \leq \frac{\epsilon}{2} ,
\quad \frac{\lvert f^*_j(l)-f_j(l) \rvert ^2}{f^*_j(l)^2} \leq \frac{\epsilon}{2k} \quad \text{ and }
\quad \frac{\lvert f^*_j(l)-f_j(l) \rvert ^2}{f_j(l)^2} \leq \epsilon 
\]
so that
\begin{equation*}
\begin{split}
&\sum_{l \in S_L} f^*_i(l)\log^2\left(\frac{f^*_j(l)}{f_j(l)} \right)
\leq \frac{\epsilon}{2},
 \sum_{l \in S_L}   \frac{\left\lvert f^*_j(l)-f_j(l) \right\rvert^2}{f^*_j(l)}
\leq \epsilon
\text{ and }
 \sum_{l \in S_L}  \frac{\left\lvert f^*_j(l)-f_j(l) \right\rvert^2}{f_j(l)}
\leq \epsilon. 
\end{split}
\end{equation*}
Moreover using Assumptions  \eqref{eq:dislog2} and \eqref{eq:disf} if for all $1 \leq i,j \leq k$,
\[ \sum_{l \in S_L^c} f^*_i(l)\log^2\left(f_j(l) \right)<\frac{\epsilon}{8},\]
then
\[
\sum_{l \in S_L^c} f^*_i(l)\log^2\left(\frac{f^*_j(l)}{f(l)} \right) 
\leq 2\sum_{l \in S_L^c} f^*_i(l) \log^2 (f^*_j(l)) +2 \sum_{l \in S_L^c} f^*_i(l)  \log^2 (f_j(l))
\leq \frac{\epsilon}{2}. 
\]
Combining the two last remarks,
 we obtain
\begin{equation*}
\begin{split}
P_G
& \geq
 \prod_{j=1}^k  \Bigg(  DP(G)\bigg(f ~ : ~\left(1-\sqrt{\frac{\epsilon}{4}} \right)f^*_j(l) \leq f(l) \leq \left(1+\sqrt{\frac{\epsilon}{4}}\right)f^*_j(l) , ~ \forall l \in S_L ,\\
& \hspace{1cm}
 \sum_{l\in S_L^c} f^*_i(l)\log^2\left(f(l) \right)  \leq \frac{\epsilon}{8}
 \text{ and }  \sum_{l\in S_L^c} f_j(l) \leq \epsilon \bigg)\bigg) \\
& \geq   \prod_{j=1}^k  \bigg(  DP(G)\bigg(f ~ : ~
\exp{\left(-\sqrt{\frac{\epsilon}{ 16  \max_{1\leq i\leq k} \sum_{l\in S_L^c}f^*_i(l)}}\right)} \leq \sum_{m \in S_L^c} f(m) \leq \epsilon ,\\
  & \qquad \sum_{l\in S_L^c} f^*_i(l)\log^2\left(\frac{f(l)}{\sum_{m \in S_L^c} f(m)} \right) \leq \frac{\epsilon}{32}, ~ \forall 1 \leq i \leq k  \text{ and }\\
& \qquad
\left(1-\sqrt{\frac{\epsilon}{4}} \right)\frac{f^*_j(l)}{\displaystyle{\sum_{m \in S_L} f(m)}} 
\leq \frac{f(l)}{\displaystyle{\sum_{m \in S_L} f(m)}}  
\leq \left(1+\sqrt{\frac{\epsilon}{4}}\right) \frac{f^*_j(l)}{\displaystyle{\sum_{m \in S_L}} f(m)} , ~ \forall l\in S_L \bigg)\Bigg)\\
\end{split}
\end{equation*}
indeed if 
\[\exp{\left(-\sqrt{\frac{\epsilon}{ 16  \max_{1\leq i\leq k} \sum_{l\in S_L^c}f^*_i(l)}}\right)} \leq \sum_{m \in S_L^c} f(m) \]
 and 
 \[\sum_{l\in S_L^c} f^*_i(l)\log^2\left(\frac{f(l)}{\sum_{m \in S_L^c} f(m)} \right) \leq \frac{\epsilon}{32}, ~ \forall 1 \leq i \leq k \]
then
\begin{equation*}
\begin{split}
\sum_{l\in S_L^c}& f^*_i(l)\log^2\left(f(l) \right) \\
&\leq 2 \sum_{l\in S_L^c} f^*_i(l)\log^2\left(\frac{f(l)}{\sum_{m \in S_L^c} f(m)} \right)
     + 2 \sum_{l\in S_L^c} f^*_i(l)\log^2\left(\sum_{m \in S_L^c} f(m)\right) 
  \leq    \frac{\epsilon}{8}.
\end{split}
\end{equation*}
Using the tail free property of the Dirichlet process, $(f(l)/\sum_{m \in S_L^c} f(m))_{l\in S_L^c}$, $\sum_{m \in S_L^c} f(m)$ and   $(f(1)/\sum_{m \in S_L} f(m),\dots,$ $f(L)/\sum_{m \in S_L} f(m))$   are independent. So that we obtain
\begin{equation}\label{eq:KL2}
\begin{split}
P_G &\geq \int_{\exp{\left(-\sqrt{\frac{\epsilon}{16 \max_{1\leq i\leq k} \sum_{l\in S_L^c}f^*_i(l)}}\right)}}^{\epsilon} 
\frac{\Gamma(G(\mathbb{N}))}{\Gamma(G(S_L)) \Gamma(G(S_L^c))} a^{G(S_L^c)-1}(1-a)^{G(S_L)-1} 
   \\
& \hspace{0.5cm} \underbrace{Dir(G|_{S_L})\bigg(x ~ : ~
\left(1-\sqrt{\frac{\epsilon}{4}} \right)\frac{f^*_j(l)}{1-a} \leq x_l \leq \left(1+\sqrt{\frac{\epsilon}{4}}\right)\frac{f^*_j(l)}{1-a} , ~ \forall l\in S_L \bigg)}_{D(a)} \lambda(da) \\
& \hspace{0.5cm}  DP(G|_{ S_L^c})\bigg(\sum_{l \in S_L^c} f^*_i(l)\log^2\left(\frac{f(l)}{\sum_{m \in S_L^c} f(m)} \right)<\frac{\epsilon}{32}, ~ \forall 1 \leq i \leq k   \bigg)
\end{split}
\end{equation}

We first control the integral of Equation \eqref{eq:KL2}.
Note that if 
\[x \in V_{SL} :=\left\{ x \in \Delta_{\lvert S_L\rvert}: ~ x_l\in V_l, ~ \forall l \in S_L \setminus \{l^*\} \right\} \]
where for all $l \in S_L \setminus \{l^*\}$ 
\[V_l:=\left\{x_l: ~ \left(1-\sqrt{\frac{\epsilon}{16}}f^*_j(l^*)\right) \frac{f^*_j(l)}{1-a} \leq x_l \leq \left(1+\sqrt{\frac{\epsilon}{16}}f^*_j(l^*)\right)\frac{f^*_j(l)}{1-a}\right\}\]
and if
\begin{equation}\label{eq:borne_a}
 a \in V_A := \left\{ \sum_{l \in S_L} f^*_j(l) - \sqrt{\frac{\epsilon}{16}} f^*_j(l^*) \leq 1-a \leq \sum_{l \in S_L} f^*_j(l) + \sqrt{\frac{\epsilon}{16}} f^*_j(l^*) \right\} \end{equation}
then for all $l \in S_L$,
\begin{equation}\label{eq:borne_aautre}
\left(1-\sqrt{\frac{\epsilon}{4}}\right)\frac{f^*_j(l)}{1-a} \leq x_l \leq \left(1+\sqrt{\frac{\epsilon}{4}}\right)\frac{f^*_j(l)}{1-a},  
\end{equation}
where $x_{l^*}=1-\sum_{l \in S_L, l\neq l^*}x_l$.
So that
\begin{equation}\label{eq:dir_sL}
\begin{split}
& D(a)
 \geq \frac{\Gamma(G( S_L))}{\prod_{m\in S_L} \Gamma(G(m))} \mathds{1}_{V_A}(a)
 \int_{V_{SL}}  \prod_{l\in S_L \setminus \{l^*\} } x_l^{G(l)-1}  \left( 1- \sum_{m \in S_L\setminus \{l^*\}} x_m \right)^{G(l^*)-1} \lambda(dx)
\end{split}
\end{equation}
where 
\begin{equation}\label{eq:discret_xl*}
\left( 1- \sum_{m\in S_L \setminus \{l^*\}} x_m \right)^{G(l^*)-1} \geq  \left(\frac{f^*_j(l^*)}{1-a} \right)^{G(l^*)-1}  \min\left((1/2)^{G(l^*)-1}, (3/2)^{G(l^*)-1}\right)
\end{equation}
 and
\begin{equation}\label{eq:intxl}
\begin{split}
& \int_{V_l}
 x_l^{G(l)-1}   \lambda(dx_{l}) \geq \left( \frac{f^*_j(l)}{1-a} \right)^{G(l)}   f^*_j(l^*) \sqrt{\frac{\epsilon}{4}}   \min\left((1/2)^{G(l)-1}, (3/2)^{G(l)-1}\right)
,\end{split}
\end{equation}
using Equation \eqref{eq:borne_aautre}.
Then combining Equations \eqref{eq:dir_sL}, \eqref{eq:discret_xl*} and \eqref{eq:intxl}; $D(a)$ is bounded by 
\begin{equation*}\label{eq:dirsLder}
\begin{split}
& \mathds{1}_{V_A}(a)  \frac{\Gamma(G(S_L))}{\prod_{m \in S_L} \Gamma(G(m))}   \left( \frac{2}{3} \right)^{L}   \left( \frac{1}{2} \right)^{G(\mathbb{N})}   
 \left( \frac{ \epsilon}{4} \right)^{(L-1)/2}  
 \left( \frac{1}{1-a} \right)^{G( S_L) -1}
 f^*_j(l^*)^{L-2}
 \prod_{l\in S_L}  f^*_j(l)^{G(l)}  .
\end{split}
\end{equation*}
So that 
\begin{equation}\label{eqmoche}
\begin{split}
&\int_{\exp{\left(-\sqrt{\frac{\epsilon}{16 \max_{1\leq i\leq k} \sum_{l\in S_L^c}f^*_i(l)}}\right)}}^{\epsilon} 
\frac{\Gamma(G(\mathbb{N}))}{\Gamma(G( S_L)) \Gamma(G(S_L^c))} a^{G(S_L^c)-1}(1-a)^{G(S_L)-1}  D(a) \lambda(da) \\
& \geq \frac{\Gamma(G(\mathbb{N}))}{G(S_L^c) \Gamma(G(S_L^c))\prod_{l\in S_L}\Gamma(G(l))} \left(  \frac{\epsilon}{4} \right)^{(L-1)/2}  \left( \frac{2}{3} \right)^{L}   \left( \frac{1}{2} \right)^{G(\mathbb{N})}   
 f^*_j(l^*)^{L-2} \prod_{l\in S_L} f^*_j(l)^{G(l)}\\
& \hspace{0.5cm} \left(\left(
\min \left(  \sum_{m \in S_L^c} f^*_j(m) + \sqrt{\frac{\epsilon}{16}} f^*_j(l^*) , \epsilon \right) \right)^{G(S_L^c)}\right. \\
& \hspace{1cm}
\left. - \left(\max\left(
\exp{\left(-\sqrt{\epsilon / \Big(16 \max_{1\leq i\leq k} \sum_{l \in S_L^c}f^*_i(l) \Big) } \right)},  
 \sum_{m \in S_L^c} f^*_j(m) - \sqrt{\frac{\epsilon}{16}} f^*_j(l^*) \right) \right)^{G(S_L^c)}
 \right) \\
 &\gtrsim   \left( \sqrt{\frac{\epsilon}{4}}  \right)^{L-1+G(\mathbb{N})}  
 \left(\frac{1}{3}\right)^L
 f^*_j(l^*)^{L-2 }  
 \prod_{l\in S_L} G(l) f^*_j(l)^{G(l)}  
.
\end{split}
\end{equation}
using that for all $0<a<1$,
\begin{equation}\label{eq:gamma}
\frac{1}{a} \leq \Gamma(a)   \leq \frac{2}{a}.
\end{equation}
and that under Assumption \eqref{eq:disf}, for $\epsilon$ small enough
\begin{equation*}
\exp\left(-\frac{\epsilon^{-\delta/2}}{4}\right) \geq 
\exp{\left(-\sqrt{\frac{\epsilon}{16 \max_{1\leq i\leq k} \sum_{l\in S_L^c}f^*_i(l)}}  \right)}
\geq 0\geq
   \sum_{m \in S_L^c} f^*_j(m) - \sqrt{\frac{\epsilon}{16}} f^*_j(l^*)  .
\end{equation*}

We now control the last term of Equation \eqref{eq:KL2}. Using Markov's inequality,
\begin{equation}\label{eq:dp_esp}
\begin{split}
& DP(G|_{S_L^c})\left(\sum_{l \in S_l^c} f^*_i(l)\log^2\left(\frac{f(l)}{\sum_{m\in S_L^c} f(m)} \right)<\frac{\epsilon}{16 }, ~ \forall 1 \leq i \leq k  \right)  \\
&\geq 1-  DP(G|_{S_L^c})  \left(\sum_{i=1}^k\sum_{l\in S_L^c} f^*_i(l)\log^2\left(\frac{f(l)}{\sum_{m\in S_L^c} f(m)} \right) >\frac{\epsilon}{16 } \right) \\
& \geq 1- \frac{  \mathbb{E}^{\theta^*}\left(\sum_{i=1}^k\sum_{l \in S_L^c} f^*_i(l)\log^2\left(\frac{f(l)}{\sum_{m\in S_L^c} f(m)}\right)\right)}{\frac{\epsilon}{16 }}  . 
\end{split}
\end{equation}
As $f(l)/\sum_{m\in S_L^c} f(m)$ is distributed from $\beta\big(G(l),G( S_L^c\setminus \{ l\}) \big)$,
\begin{equation}\label{eq:Elog2fl}
\begin{split}
&\mathbb{E}^{\theta^*}\left[ \log^2\left( \frac{f(l)}{\sum_{m\in S_L^c} f(m)}\right)\right] \\
&=
   \frac{\Gamma( G(S_L^c))}{\Gamma( G(l))\Gamma( G(S_L^c \setminus\{l\}))}
 \quad   \Bigg(\underbrace{\int_0^{1/2}  \log^2(x) x^{G(l)-1} (1-x)^{G(S_L^c \setminus\{l\})-1} \lambda(dx) }_{I_1}   \\
& \hspace{4.3cm}+
\underbrace{\int_{1/2}^1  \log^2(x) x^{G(l)-1} (1-x)^{G(S_L^c \setminus\{l\})-1}  \lambda(dx) }_{I_2}    \Bigg),
\end{split}
\end{equation}
with
\begin{equation}\label{eq:I1}
\begin{split}
I_1 
&\leq 2  \int_0^{1/2}  \log^2(x) x^{G(l)-1} \lambda(dx) 
= \frac{4 \log^2(2) (1/2)^{G(l)}}{G(l)^3}  \left(   \frac{G(l)^2}{2}  + \frac{G(l)}{\log 2}   + \frac{1}{\log^2 2}  \right)\\
&\leq \frac{12 \log^2 (2) (G(\mathbb{N}))^2}{G(l)^3}
,\end{split}
\end{equation}
and
\begin{equation}\label{eq:I2}
\begin{split}
I_2  
\leq  2\log^2(2)  \int_{1/2}^1  (1- x)^{G(S_L^c \setminus\{l\})-1}  \lambda(dx)
\leq \frac{2 \log^2(2)}{G(S_L^c \setminus\{l\})}.
\end{split}
\end{equation}
Combining Equations \eqref{eq:gamma} \eqref{eq:Elog2fl}, \eqref{eq:I1} and \eqref{eq:I2},
we obtain
\begin{equation}\label{eq:espe}
 \mathbb{E}^{\theta^*}\left(\sum_{l \in S_L^c} f^*_i(l)\log^2\left(\frac{f(l)}{\sum_{m\in S_L^c} f(m)}\right)\right) 
\leq 24 G(\mathbb{N})^2 \log^2(2) \sum_{l \in S_L^c} \frac{f^*_i(l)}{(G(l))^2} .
\end{equation}
Then using Assumption \eqref{eq:disf_g} and Equations \eqref{eq:dp_esp} and \eqref{eq:espe}
\begin{equation}\label{eqmoche2}
\begin{split}
& DP(G|_{S_L^c})
\bigg(\sum_{l \in S_L^c} f^*_i(l)\log^2\left(\frac{f(l)}{\sum_{m\in S_L^c} f(m)} \right)<\frac{\epsilon}{16 k}, ~ \forall 1 \leq i \leq k \bigg) \geq 1/2
\end{split}
\end{equation}
Lemma \ref{prop:discret} follows combining Equations \eqref{eq:KL2}, \eqref{eqmoche} and \eqref{eqmoche2}.
\end{proof}


\section{Proof of Theorem \ref{th:vit_cont_vrai} (Dirichlet process mixtures of Gaussian distributions)}\label{se:pr_th_continuous}

\begin{proof}[Proof of Theorem \ref{th:vit_cont_vrai} ]
Let $\sigma_n=\tilde{\epsilon}_n /(\log (1/ \tilde{\epsilon}_n))$,
 $\tilde{\epsilon_n}=n^{-\beta/(2\beta +1)} (\log n)^{t_0}$. Following the computations of the proof of Theorem 4 of \citep{ShToGh13} and using Assumption \ref{hyp:B10}, Lemma \ref{th:vit_cont} ensures that Assumption \eqref{hyp:KL} holds with $t_0 \geq (2+2/\gamma_0 +1)/(1/\beta +2) $. Using Assumption \eqref{hyp:Q0}, Assumption \eqref{hyp:voisQ*} holds. Using Theorem 5 of \citep{ShToGh13}, Assumptions \ref{hyp:B7}, \ref{hyp:B8} and \ref{hyp:B9}; Assumption \eqref{hyp:Fn} holds with $\epsilon_n=n^{-\beta /(2 \beta +1)} (\log(n))^t$, $t>t_0$. This concludes the proof of Theorem \ref{th:vit_cont_vrai}.
 \end{proof}

The following lemma is a  generalization of Lemma 4 of \citep{KrRoVa10} in the HMM context. In other words, we give a set of density functions $(f_j)_{1\leq j \leq k}$ satisfying Assumptions \eqref{hypa:log2}--\eqref{hypa:qui2S} in Lemma \ref{th:vit_cont}.

\begin{lemma}\label{th:vit_cont}
Assume that there exist $\beta,L$ and $\gamma$ such that for all $1\leq j \leq k $, $f^*_j \in \mathcal{P}(\beta,L,\gamma)$ and Assumptions  \ref{hyp:B2}--\ref{hyp:B4} hold. Let $\sigma$ be a positive real small enough.

Then for all $1 \leq j \leq k$, there exists a discrete measure 
 $m_j= \sum_{i=1}^{N_j} \mu_j^i \delta_{z_j^i} $ supported on $\{ x  ~ : ~ f^*_j(x) \geq K_j \sigma^{2\beta +H_1}\}$ with $H_1> 2 \beta$,  $K_j$ a constant small enough 
 and $N_j=O(\sigma^{-1} \lvert \log \sigma \rvert ^{2/\gamma_0})$  
 such that
Assumptions \eqref{hypa:log2}--\eqref{hypa:qui2S} hold with $f_j=\phi_\sigma * m_j$ for all $1 \leq j \leq k$ and $\sigma^{2\beta}\leq \tilde{\epsilon}_n^{2}/u_n$.

Assumptions \eqref{hypa:log2}--\eqref{hypa:qui2S} also hold with $f_j=\phi_{\tilde{\sigma}} *\tilde{m}_j$, for all $\tilde{\sigma} \in [\sigma, _ \sigma + \sigma^{\delta' H_1 +2}]$  and for all $\tilde{m}_j= \sum_{i=1}^{+\infty} \tilde{\mu}_j^i \delta_{\tilde{z}_j^i} $ such that $\tilde{\mu}_j^{1:N_j} \in \mathcal{B}(\mu_j, \sigma^{\delta' H_1 +2})$ and   $\tilde{z}_j^i \in \mathcal{B}(z_j^i, \sigma^{\delta' H_1 +2})$, for all $1\leq i \leq N_j$, where $\delta' \geq 1+ \beta/H_1$.
\end{lemma}

\begin{proof}[Proof of Lemma \ref{th:vit_cont}]
The proof of Lemma \ref{th:vit_cont} is based on \citep{KrRoVa10}. First notice that $f^*_j \in \mathcal{P}(\beta,L,\gamma)$ implies that for all integer $m\leq \beta$, $\lvert l^j_m \rvert$ is bounded by a polynomial. Then Assumption  \ref{hyp:B2} implies that Assumption (C2) of \citep{KrRoVa10} holds and stronglier implies that there exists $\delta>0$ such that for all $1 \leq i,j,\iota \leq k$ and all integer $m\leq \beta$,
 \begin{equation}\label{eq:assC2}
 \begin{split}
 &\int \lvert l^j_m(x) \rvert^{(2 \beta + \delta)/m} f^*_i(x)  \lambda(dx) < \infty,
\\
&\int \lvert L^j(x) \rvert^{2  + \delta/\beta} f^*_i(x)  \lambda(dx) < \infty,
\\
 &\int \lvert l^j_m(x) \rvert^{(2 \beta + \delta)/m} f^*_i(x) \log f^*_\iota (x) \lambda(dx) < \infty,
\\
&\int \lvert L^j(x) \rvert^{2  + \delta/\beta} f^*_i(x) \log f^*_\iota (x) \lambda(dx) < \infty.
 \end{split}
 \end{equation}

Let $\sigma >0$, we consider 
\[
S=\bigcap_{j=1}^k \left( A_\sigma^j \cap E_\sigma^j \right),
\]
where 
\[
 A_\sigma^j = \left\{ x ~ : ~ \lvert l_m^j(x) \rvert \leq B \sigma ^{-m} \lvert \log(\sigma) \rvert^{-m/2}, \forall 1 \leq m \leq \lfloor \beta \rfloor,  ~ \lvert L^j(x) \rvert \leq B \sigma^{- \beta}\lvert \log(\sigma) \rvert^{-\beta/2}  \right\}
 \]
and
 \[
 E_\sigma^j=\left\{ x  ~: ~ f^*_j(x)\geq  \sigma^{H_1} \right\} .
 \]
Using Assumptions \ref{hyp:B1}, \ref{hyp:B2}, \ref{hyp:B4} and Lemma 2 of  \citep{KrRoVa10}, there exists $k$ density functions $h_\beta^j$ such that for all $1 \leq j \leq k$ and all $x \in S$,
\begin{equation}\label{eq:approxf*}
h_\beta^j * \phi_\sigma(x) = f^*_j(x)\left(1 + O(R^j(x) \sigma^\beta) \right) + O((1+R^j(x))\sigma^H) 
\end{equation}
where $R^j$ is defined as in Equation (16) page 1232 of \citep{KrRoVa10} and $H$ is as large as we want.
Using Assumptions \ref{hyp:B1}, \ref{hyp:B2} and \ref{hyp:B4}, the proof of Lemma 2 of \citep{KrRoVa10} is easily generalizable in this context so that
\begin{equation}\label{eq:tilde_f_sc}
\begin{split}
\int_{S^c} h_\beta^j*\phi_\sigma(y)\lambda(dy) \lesssim \sigma^{2\beta}
.
\end{split}
\end{equation} The generalization can be proved using  Equation \eqref{eq:assC2} and by replacing Equation (56) of \citep{KrRoVa10} by Equation \eqref{eq:Ef*}.

 As at page 1251 in  \citep{KrRoVa10}, we denote
\[
\tilde{h}^\beta_j=\frac{\mathds{1}_{x : h^\beta_j(x) \geq \sigma^{H_2}} h^\beta_j}{\int_{x : h^\beta_j(x) \geq \sigma^{H_2}} h^\beta_j d\lambda} 
\]
and using Lemma 12 of  \citep{KrRoVa10}, for all $1 \leq j \leq k$, there exist 
$k$ 
 discrete distributions
  $m_j= \sum_{i=1}^{N_j} \mu_j^i \delta_{z_j^i} $ 
  supported in
   $\{ x  ~ : ~ f^*_j(x) \geq K_j \sigma^{2\beta +H_1}\}$
    for 
    $H_1> 2 \beta$, 
    a constant $K_j$ 
    small enough and with 
    $N_j=O(\sigma^{-1} \lvert \log \sigma \rvert ^{2/\gamma_0})$  
such that 
\begin{equation}\label{eq:borne_discret}
\begin{split}
& \lVert \tilde{h}^\beta_j * \phi_\sigma - m_j* \phi_\sigma \rVert_\infty \leq \sigma^{-1} e^{-C\lvert \log \sigma \rvert ^{2/\gamma_0} } ,\\
&
\lVert\tilde{h}^\beta_j * \phi_\sigma - m_j* \phi_\sigma \rVert_1 \leq \sigma^{-1} e^{-C'\lvert \log \sigma \rvert ^{2/\gamma_0} }
\end{split}
\end{equation}
for any $C',C$ large enough.

It is now sufficient to prove that Assumptions \eqref{hypa:log2} to \eqref{hypa:qui2S} hold with $\epsilon^{4/(2-\alpha)}_n=O(\sigma^{2\beta})$, 
$f_j=\tilde{m}_j * \phi_{\tilde{\sigma}}$ 
and
$\tilde{f}_j=h^\beta_j * \phi_{\tilde{\sigma}}$
for all 
$\tilde{\sigma} \in [\sigma, \sigma + \sigma^{\delta' H_1 +2}]$
 and all discrete distributions 
 $\tilde{m}_j= \sum_{i=1}^{N_j} \tilde{\mu}_j^i \delta_{\tilde{z}_j^i} $
  such that 
  $\tilde{\mu}_j \in \mathcal{B}(\mu_j, \sigma^{\delta' H_1 +2})\cap \Delta_{N_j}$
   and 
   $\tilde{z_j} \in \mathcal{B}(z_j, \sigma^{\delta' H_1 +2})$
    where 
    $\delta' \geq 1+ \beta/H_1$.
By Lemma 2 of \citep{KrRoVa10},
\begin{equation}\label{eq:bonre_l1_linf_f_m}
\lVert f_j - m_j*\phi_{\sigma} \rVert_1 \lesssim \sigma^{H_1 + 1 + \beta}, 
\quad 
\lVert f_j - m_j*\phi_{\sigma} \rVert_\infty \lesssim \sigma^{H_1 + \beta}, 
\end{equation}

\begin{itemize}
\item Proof of \eqref{hypa:log2}. We cut the integral as in the following,
\begin{equation}\label{eq:contKLi}
\begin{split}
\int  f^*_i(y)\log^2 \frac{f^*_j(y)}{f_j(y)} \lambda(dy)
\leq &
 \int_S f^*_i(y)\log^2 \frac{f^*_j(y)}{\tilde{f}_j(y)} \lambda(dy) \\
& +
\int_{S^c} f^*_i(y)\log^2 \frac{f^*_j(y)}{\tilde{f}_j(y)} \lambda(dy) \\
& +
\int f^*_i(y)\log^2 \frac{\tilde{f}_j(y)}{f_j(y)} \lambda(dy)
.\end{split}
\end{equation}
The last integral can be controlled by $O(\sigma^{2\beta})$ as in the proof of Lemma 4 of \citep{KrRoVa10}.

Using Equation \eqref{eq:approxf*} we control the first integral of the bound of Equation \eqref{eq:contKLi}:
\begin{equation*}
\begin{split}
&\int_S f^*_i(y)\log^2 \frac{f^*_j(y)}{\tilde{f}_j(y)} \lambda(dy)
\leq
\int_S f^*_i(y) \frac{\lvert f^*_j(y)-\tilde{f}_j(y)\rvert^2}{\min(\tilde{f}_j(y),f^*_j(y))} \lambda(dy) \lesssim
\sigma^{2 \beta}
\end{split}
\end{equation*}
as soon as $H$ is large enough, using Equation \eqref{eq:approxf*}.

Using that $f^*_j \lesssim \tilde{f}_j \lesssim \phi_\sigma*f^*_j \lesssim 1$ (see Remark 1  and the bottom of page 1252 of \citep{KrRoVa10}), we control the second integral in the bound of Equation \eqref{eq:contKLi},
\begin{equation*}
\begin{split}
& \int_{S^c} f^*_i(y)\log^2 \frac{f^*_j(y)}{\tilde{f}_j(y)} \lambda(dy) \lesssim
 \int_{S^c} f^*_i(y) \lambda(dy) 
 +
  \int_{S^c} f^*_i(y) \log^2 \left(f^*_j(y)\right) \lambda(dy) 
\end{split}
\end{equation*}
which is bounded by $O(\sigma^{2 \beta})$ following the proof of \eqref{hypa:maxf*Sc}.

\item Proof of \eqref{hypa:qui2}.
We cut the integral into three parts:
\begin{equation*}
\begin{split}
&  \int_S \frac{  \lvert f^*_i(y) - f_i(y)  \rvert^2}{f^*_i(y)}  \lambda(dy)  \\
&  \lesssim  
\underbrace{\int_S \frac{  \lvert f^*_i(y) - h_\beta^i*\phi_\sigma(y)  \rvert^2}{f^*_i(y)}  \lambda(dy)}_{I_1}
+ \underbrace{ \int_S \frac{  \lvert h_\beta^i*\phi_\sigma(y) - \tilde{h}_\beta^i*\phi_\sigma(y)  \rvert^2}{f^*_i(y)}  \lambda(dy)  }_{I_2} \\
&\quad + \underbrace{ \int_S \frac{  \lvert \tilde{h}_\beta^i*\phi_\sigma(y) - f_i(y)  \rvert^2}{f^*_i(y)}  \lambda(dy) }_{I_3} \\
\end{split}
\end{equation*}
Using Equation \eqref{eq:approxf*},
\begin{equation}\label{eq:qui2_S_tilde_f}
I_1 \lesssim \sigma^{2\beta}.
\end{equation}
We now control $I_2$ using the bound
\[
\left\lvert \frac{h_\beta^i*\phi_\sigma(y)}{\tilde{h}_\beta^i*\phi_\sigma(y)} -  1\right\rvert =O(\sigma^{2\beta}), \text{ for all } y \in S
\]
 of page 1252 of \citep{KrRoVa10} and Equation \eqref{eq:approxf*}. Then
\begin{equation}
\begin{split}
I_2& = \int_S \left(\frac{  \lvert h_\beta^i*\phi_\sigma(y) - \tilde{h}_\beta^i*\phi_\sigma(y)  \rvert}{h_\beta^i*\phi_\sigma(y)}\right)^2 h_\beta^i*\phi_\sigma(y) \frac{h_\beta^i*\phi_\sigma(y)}{f^*_i(y)} \lambda(dy)\\
& 
\leq \int_S (O(\sigma^{2\beta}))^2 ~  h_\beta^i*\phi_\sigma(y) ~  2 ~  \lambda(dy) \lesssim \sigma^{2 \beta}
,\end{split}
\end{equation}
using Equation \eqref{eq:approxf*}.
As to $I_3$, using Equations \eqref{eq:borne_discret} and \eqref{eq:bonre_l1_linf_f_m}, it is upper-bounded by
\begin{multline}
2 \frac{\lVert \tilde{h}^\beta_i * \phi_\sigma - m_i*\phi_\sigma \rVert_\infty  
\lVert \tilde{h}^\beta_i *  \phi_\sigma - m_i*\phi_\sigma \rVert_1}{\sigma^{H_1}} 
  + 2 \frac{\lVert f_i - m_i*\phi_\sigma \rVert_\infty  
\lVert f_i - m_i*\phi_\sigma \rVert_1}{\sigma^{H_1}} 
\lesssim \sigma^{2\beta}
\end{multline}
when $ 2 > \gamma_0 $, such a $\gamma_0$ can always be chosen (see the first line of page 1253 of \citep{KrRoVa10}).

\item Proof of \eqref{hypa:maxSc}. Assumption \eqref{hypa:maxSc} is proved in Equation \eqref{eq:tilde_f_sc}.

\item Proof of \eqref{hypa:maxf*Sc}. It is sufficient to bound 
\[
 \int_{(E_\sigma^i)^c}  f^*_j(y)   \lambda(dy) 
 \]
and
\[
 \int_{(A_\sigma^i)^c}  f^*_j(y)   \lambda(dy) .
 \]
Using Assumption \ref{hyp:B3}, for all $0<\delta<1$
\begin{equation}\label{eq:Ef*}
\begin{split}
&\int_{(E_\sigma^i)^c}  f^*_j(y)   \lambda(dy) \\
& \leq \int_{\{y : f^*_j(y)< \sigma^{H_1} M_{j,i} \exp(\tau_{j,i} \lvert y \rvert^{\gamma_{j,i}} )\}}  f^*_j(y)   \lambda(dy) \\
& \leq \int_{\{y : f^*_j(y)< \sigma^{H_1} M_{j,i} \exp(\tau_{j,i} \lvert y \rvert^{\gamma_{j,i}} )\}}  (f^*_j(y))^{1/\delta}  (f^*_j(y))^{1-1/\delta} \lambda(dy)\\
& \leq \sigma^{H_1/\delta} \int (M_{j,i})^{1/\delta} \exp(\tau_{j,i} \lvert y \rvert^{\gamma_{j,i}}/\delta)   (f^*_j(y))^{1-1/\delta} \lambda(dy) \lesssim \sigma^{2 \beta}
\end{split}
\end{equation}
as soon as $H_1 > 2 \beta$, using Assumption \ref{hyp:B2}.
Moreover using \eqref{eq:assC2} and Markov inequality, as in the proof of Lemma 2 of \citep{KrRoVa10},  
\begin{equation}
\begin{split}
& \int_{(A_\sigma^i)^c}  f^*_j(y)   \lambda(dy) \lesssim \sigma^{2\beta}.
\end{split}
\end{equation}

\item Proof of \eqref{hypa:maxloSc}. Using the same argument as in the bottom of the page 1252 of \citep{KrRoVa10}, Equations \eqref{eq:borne_discret} and \eqref{eq:bonre_l1_linf_f_m},
\begin{equation}
\begin{split}
& \int_S f^*_i(y) \max_{1 \leq j \leq k} \log \left( \frac{\tilde{f}_j(y)}{f_j(y)} \right) \lambda(dy)\\
&\leq 
\int_S f^*_i(y) \max_{1 \leq j \leq k}  \frac{\lvert \tilde{f}_j(y)-f_j(y) \rvert}{f_j(y)} \lambda(dy)\\
&\leq
\int_S f^*_i(y) \max_{1 \leq j \leq k}  \frac{\lVert \tilde{f}_j-f_j \rVert_\infty}{\sigma^{H_2}  -  \lVert \tilde{f}_j-f_j \rVert_\infty} \lambda(dy)
\lesssim \sigma^{2\beta}
\end{split}
\end{equation}

\item Proof of \eqref{hypa:qui2S}.  Using that  $f^*_j \lesssim \tilde{f}_j$ (see Assumption (C3) of \citep{KrRoVa10}) Equation \eqref{eq:qui2_S_tilde_f} implies \eqref{hypa:qui2S}.

\end{itemize}

\end{proof}

\end{appendices}

\bibliographystyle{plainnat}
\bibliography{biblio2}

\end{document}